\documentclass[11pt]{article}
\usepackage{authblk}
\usepackage{amsthm}
\usepackage{amsmath}
\usepackage{amssymb}
\usepackage{latexsym}
\usepackage{graphicx}
\usepackage{longtable}
\usepackage{color,varioref}
\usepackage{algorithm}
\usepackage{algorithmic}
\usepackage[breaklinks,colorlinks]{hyperref}

\usepackage{booktabs}
\usepackage{graphicx} 
\usepackage{subfigure}  
\usepackage{caption}  
\usepackage{array}
\usepackage{multirow}
\usepackage{array}
\usepackage{setspace}
\hypersetup{urlcolor=black}
\usepackage{enumerate}

\newtheorem{theorem}{Theorem}[section]
\newtheorem{proposition}{Proposition}[section]

\numberwithin{equation}{section}

\definecolor{lred}{rgb}{1,0.8,0.8}
\definecolor{lblue}{rgb}{0.8,0.8,1}
\definecolor{dred}{rgb}{0.6,0,0}
\definecolor{dblue}{rgb}{0,0,0.5}
\definecolor{dgreen}{rgb}{0,0.5,0.5}
\definecolor{blue}{rgb}{0,0,0.9}
\definecolor{red}{rgb}{0.9,0,0}
\definecolor{green}{rgb}{0,0.9,0}

\long\def\acks#1{\vskip 0.3in\noindent{\large\bf Acknowledgments}\vskip 0.2in
	\noindent #1}

\begin{document}
	
	\title{\bf Adaptive sieving with semismooth Newton proximal augmented Lagrangian algorithm for multi-task Lasso problems}

	\author{
	Lanyu Lin\textsuperscript{1}, 
	Yong-Jin Liu\textsuperscript{1,2}, 
	Bo Wang\textsuperscript{1}, 
	and Junfeng Yang\textsuperscript{3}
	}
	\renewcommand{\thefootnote}{\arabic{footnote}} 

	\footnotetext[1]{School of Mathematics and Statistics, Fuzhou University, Fuzhou 350108, Fujian, P. R. China (e-mail: lylin@fzu.edu.cn, bowang@fzu.edu.cn).}
	\footnotetext[2]{Corresponding author. Center for Applied Mathematics of Fujian Province, School of Mathematics and Statistics, Fuzhou University, Fuzhou 350108, Fujian, P. R. China (e-mail: yjliu@fzu.edu.cn).}
	\footnotetext[3]{Department of Mathematics, Nanjing University, Nanjing 210093, Jiangsu, P. R. China (e-mail: jfyang@nju.edu.cn).}

	\maketitle
	
	\noindent {\bf Abstract.} 
    Multi-task learning enhances model generalization by jointly learning from related tasks. This paper focuses on the $\ell_{1,\infty}$-norm constrained multi-task learning problem, which promotes a shared feature representation while inducing sparsity in task-specific parameters. We propose an adaptive sieving (AS) strategy to efficiently generate a solution path for multi-task Lasso problems. Each subproblem along the path is solved via an inexact semismooth Newton proximal augmented Lagrangian ({\sc Ssnpal}) algorithm, achieving an asymptotically superlinear convergence rate. By exploiting the Karush-Kuhn-Tucker (KKT) conditions and the inherent sparsity of multi-task Lasso solutions, the {\sc Ssnpal} algorithm solves a sequence of reduced subproblems with small dimensions. This approach enables our method to scale effectively to large problems. Numerical experiments on synthetic and real-world datasets demonstrate the superior efficiency and robustness of our algorithm compared to state-of-the-art solvers.
	\vspace{2mm}
	\vspace{2mm}

	\section{Introduction} 
	For a given set of $n$ tasks, consider an $n$-task linear regression model:
	\begin{equation}\label{mtlproblem1}
		\hat{y}^i=X^i \hat{w}^i+\hat{\epsilon}^i, \ i=1,\ldots,n,
	\end{equation}
	where $X^i\in\mathbb{R}^{m_i\times d}$ is the observed data for the $i$-th task, $\hat{y}^i\in \mathbb{R}^{m_i}$ is the response variable, the vector $\hat{w}^i\in \mathbb{R}^d$ is the estimated coefficient, and $\hat{\epsilon}^i\in \mathbb{R}^{m_i}$ indexes the vector of random noises. Here, $m_i$ denotes the number of instances for the $i$-th task, and $d$ denotes the feature dimensionality. To share information across different tasks and attain the regression coefficients $\hat{w}^i, \ i=1,\ldots,n$, we introduce the multi-task Lasso problem in the following form:
	\begin{equation}\label{mtlproblem1}
		\min_{W=[\hat{w}^1,\hat{w}^2,\ldots,\hat{w}^n]} \Big\{   \frac{1}{2} \sum_{i=1}^{n}\|\hat{y}^i-X^i \hat{w}^i\|^2 \mid \|W\|_{1,\infty} \leq \gamma\Big\},
	\end{equation}
	where $\|W\|_{1,\infty}=\sum_{j=1}^{d}\max_{1\le i\le n}\lvert W_{ji}\rvert$, and $\gamma>0$ is a positive parameter. The multi-task learning was first proposed by Caruana \cite{Caruana1998Multitask} to improve generalization performance by leveraging useful information from multiple related learning tasks and has been widely used in many applications, including speech recognition \cite{Parameswaran2010Large}, handwritten digits recognition \cite{Quadrianto2010Multitask}, and disease progression prediction \cite{Zhou2011A}. To achieve group-wise sparsity across different tasks, recent studies have demonstrated that certain mixed norms, such as the $\ell_{1,2}$-norm \cite{Argyriou2008Convex,Obozinski2006Multi,Obozinskis2008High,Tropp2006Algorithms} and the $\ell_{1,\infty}$-norm \cite{Turlach2005Simultaneous,Quattoni2009An}, serve as effective regularizers or constraints in multi-task learning. In particular, the $\ell_{1,\infty}$-norm is often more advantageous than the $\ell_{1,2}$-norm for inducing sparse solutions. Consequently, this paper focuses on the multi-task learning problem with $\ell_{1,\infty}$-norm ball constraints, also known as the multi-task Lasso problem.

	The multi-task Lasso problem emerges from a variety of research areas, such as machine learning, statistics, and signal processing. In the literature, several methods have been proposed to tackle the $\ell_{1,\infty}$-norm ball constrained multi-task learning problem. Specifically, first-order methods include the double coordinate descent method \cite{Zhang2006A}, the projected gradient (PGM) algorithm \cite{Chu2020Semismooth}, and the projected subgradient algorithm \cite{Quattoni2009An}. Second-order algorithms such as the primal-dual infeasible interior-point algorithm  \cite{Wright1997Primal,Turlach2005Simultaneous} and the inexact semismooth Newton augmented Lagrangian ({\sc Ssnal}) algorithm \cite{Lin2024An} have been designed. Additionally, researchers have proposed some gradient (or subgradient) approaches to solve the $\ell_{1,\infty}$-norm regularized multi-task learning problem, which is equivalent to the multi-task Lasso problem (\ref{mtlproblem1}) follows from the Lagrangian duality. It is worthwhile to mention that there is a one-to-one correspondence between the regularizer involved in the $\ell_{1,\infty}$-norm regularized regression and $\gamma$ in Problem (\ref{mtlproblem1}), but it is difficult to characterize it by an analytic formula (cf. \cite{Osborne2000On}). In \cite{Liu2009Blockwise}, a scalable cyclical blockwise coordinate descent algorithm is developed to evaluate the entire regularization path of the $\ell_{1,\infty}$-norm regularized multi-task learning problem. However, there is a lack of convergence analysis for the algorithm in \cite{Liu2009Blockwise}. Duchi et al. \cite{Duchi2009Efficient} present the forward-backward splitting method (FOLOS) with the convergence rate $O(1/\sqrt{t})$ to solve the multi-task learning problem with the $\ell_{1,\infty}$-norm regularizer, where $t$ is the number of iterations. Based on a sorting procedure, Chen et al. \cite{Chen2009Accelerated} propose an accelerated gradient method (AGM) with the convergence rate $O(1/t^2)$ to solve the $\ell_{1,\infty}$-norm regularized problem and the numerical results demonstrate that the AGM algorithm outperforms the PGM algorithm \cite{Quattoni2009An} and the FOLOS algorithm \cite{Duchi2009Efficient}. Unfortunately, these gradient (or sub-gradient) methods are difficult to efficiently address large scale optimization problems. Therefore, this paper aims to design highly efficient and reliable algorithms for solving large-scale multi-task Lasso problems.

	To fully leverage the inherent sparsity of solutions to the multi-task Lasso problem and solve it more efficiently, we employ the adaptive sieving (AS) strategy first proposed in \cite{Lin2020Adaptive} to obtain solution paths of the multi-task Lasso problems. It is worth noting that the AS strategy is applicable to more general regularizers compared to other screening rules \cite{Tibshirani2012Strong,Ghaoui2010Safe,Wang2013Lasso}, which yields that the AS strategy is widely used for sparsity optimization problems, such as machine learning models with the exclusive lasso regularizer \cite{Lin2020Adaptive}, the estimation of high dimensional sparse precision matrices \cite{Li2023MARS}, the $\ell_1$-Regularized Logistic Regression \cite{Liu2023Dual}, among others. The central concept of the AS strategy is to reduce the number of variables based on the KKT conditions, thereby enhancing the efficiency of the algorithms involved. For a sequence of reduced subproblems, this paper develops an inexact semismooth Newton proximal augmented Lagrangian ({\sc Ssnpal}) algorithm. This idea is inspired by the outstanding performance of the {\sc Ssnal} algorithm in addressing multi-task Lasso problems \cite{Lin2024An} and other sparsity problems \cite{Li2018A,Li2020On,Lin2019Efficient,Fang2021Efficient}. The novelty of the {\sc Ssnpal} algorithm lies in its asymptotic superlinear convergence property, as well as the low computational costs in utilizing the semismooth Newton algorithm to solve its inner problems, resulting in the impressive numerical performance. To assess the efficiency and robustness of our proposed algorithm (referred to as AS-{\sc Ssnpal}), we conduct comprehensive numerical experiments to compare its numerical performance with that of {\sc Ssnpal}, ADMM \cite{Gabay1976A,Glowinski1975Sur}, and AS strategy with ADMM (referred to as AS-ADMM).

	The remaining parts of this paper are briefly summarized as follows. Section \ref{Projector} derives an explicit formula for constructing the generalized HS-Jacobian of the projector onto $\ell_{1,\infty}$-norm ball. In Section \ref{sec:AS}, we introduce the AS strategy to generate solution paths of multi-task Lasso problems and characterize its convergence properties. Section \ref{sec:Ssnpal} presents an inexact semismooth Newton proximal augmented Lagrangian algorithm for the reduced problems in the AS strategy. The {\sc Ssnpal} algorithm utilizes the semismooth Newton algorithm, known for its relatively low computational cost and fast superlinear or even quadratic convergence rates, to optimize its inner problems. We also establish global and asymptotically superlinear local convergence results for the {\sc Ssnpal} algorithm under mild conditions. In Section \ref{sec:experiments}, we conduct extensive numerical experiments on multi-task Lasso problems and demonstrate that our proposed algorithm is more effective and reliable than the {\sc Ssnpal} algorithm, the AS-ADMM algorithm, and the ADMM algorithm. Section \ref{sec:conclude} provides the concluding remarks of the paper.
	
	{\bf Notation.}
	Given positive integers $d$ and $n$, we denote $\mathbb{R}^n$ as the space of all $n$-dimensional vectors, $\mathbb{R}^{d\times n}$ as the space of all $d\times n$ real matrices, $I_n$ as the $n\times n$ identity matrix, and ``${\rm rem}(d,n)$" as the remainder of dividing $d$ by $n$. For any $x\in\mathbb{R}^n$, ${\rm sign}(x)$ refers to the sign vector with each component ${\rm sign}(x_i)$, $i=1,2,\ldots,n$, defined by ${\rm sign}(x_i)$=1 if $x_i>0$, $-1$ if $x_i<0$, and $0$ otherwise. ${\rm Diag}(x)$ is the diagonal matrix constructed with the vector $x$ as its diagonal elements. For a set $\mathcal{I}\subseteq \mathbb{R}^n$, we denote its cardinality by $\lvert\mathcal{I}\rvert$, and the distance of $x$ to $\mathcal{I}$ is given by ${\rm dist}(x,\mathcal{I}):={\rm inf}_{x'\in \mathcal{I}}\|x'-x\|$. Let $X\in \mathbb{R}^{d\times n}$ be a given matrix. $X^{\dagger}$ is the Moore-Penrose inverse of $X$. The vector ${\rm Vec}(X)$ is obtained by stacking up the columns of matrix $X$. Conversely, the operator ``Mat" serves as the inverse of ``Vec", meaning that $X={\rm Mat}({\rm Vec}(X))$.
	When $\mathcal{I}\subseteq \mathbb{R}^n$ is an index set, $X_{[d]{\mathcal{I}}}$ denotes the submatrix of $X$ formed by selecting the columns of $X$ indexed by $\mathcal{I}$, and a similar convention applies to $X_{{\mathcal{I}}[n]}$, where ``$[n]$" is often omitted for simplicity. The Fenchel conjugate of a proper convex function $f$ is denoted by $f^*$. The Moreau envelope and the proximal mapping of $f$ \cite{Moreau1965Proximite} are given by 
	\begin{align*}
		\mathcal{H}_{f}(x) &:= \min_{y\in \mathbb{R}^n} \{ \frac{1}{2}\|y-x\|^2+f(y)\},\\
		{\rm Prox}_{f}(x) &:= \underset{y\in \mathbb{R}^n}{\rm arg \min} \{ \frac{1}{2}\|y-x\|^2+f(y)\}. 
	\end{align*}

	\section{The metric projection onto the $\ell_{1,\infty}$-norm ball}\label{Projector}
	In this section, we shall review some attractive results on the projection onto the $\ell_{1,\infty}$-norm ball, which will be used in the subsequent text.
	
	Given a matrix $Q\in \mathbb{R}^{d\times n}$ and positive parameter $\gamma>0$, our aim is to reformulate the following projection problem
	\begin{equation}\label{prob:projection}
		\min_{W} \Big\{ \frac{1}{2} \|W-Q\|_F^2 \mid \|W\|_{1,\infty} \leq \gamma \Big\}.
	\end{equation}
	For convenience, we first introduce some notations. For each $i=1,2,\ldots,d$, let $\lambda_i>0$ be the smallest integer $k\in \{1,2,\ldots,n\}$ such that
	\begin{equation*}
		\lvert Q_{ik} \rvert \geq \lvert Q_{ij}\rvert , \, \ j = 1,2,\ldots,n.
	\end{equation*}	
	Denote the matrix $T\in \mathbb{R}^{d\times n}$ by
	\begin{equation}\label{equa:T}
		T_{ij}:=
		\left\{ 
		\begin{array}{ll}
			{\rm sign}(Q_{ij}), & j = \lambda_i,  \\ [5pt]
			0,& {\rm otherwise}.
		\end{array} \right.
	\end{equation} 	
	Let $E:=[e^1,e^2,\ldots,e^d]^\top \in \mathbb{R}^{d\times nd}$, where for each $i = 1,\ldots,d$, $e^i =[e_1^i,\ldots,e_{nd}^i]^\top\in \mathbb{R}^{nd}$ is defined by
	\begin{equation*}
		e_k^i:=
		\left\{ 
		\begin{array}{ll}
			T_{i \lambda_{i}}, & k = \lambda(i),  \\ [5pt]
			0,& {\rm otherwise}, 
		\end{array} \right. 
		\ k = 1,\ldots,nd,
	\end{equation*} 
	where $\lambda(i)=(\lambda_i-1)d+i$. Let $w={\rm Vec}(W)$,  $q={\rm Vec}(Q)$, $b={\rm Vec}(T)$, and $D=F-\mathbf{1}_n\otimes E$ with $F = {\rm Diag}({\rm sign}(q))$, where $\mathbf{1}_n$ denotes the vector of all ones, and $\otimes$ is the Kronecker product. Invoking \cite[Proposition 2.1]{Lin2024An}, 
	one can readily rewrite Problem (\ref{prob:projection}) as
	\begin{equation}\label{prob:vec_projection1}
		\min_{w} \Big\{ \frac{1}{2} \|w-q\|^2 \mid b^\top w \le \gamma, \ Dw\le 0, \ Fw\ge0\Big\}.
	\end{equation}
	It is obvious that the feasible set of Problem (\ref{prob:vec_projection1}) is nonempty and polyhedral convex.

	Next, we want to present a nontrivial formula for the generalized HS-Jacobian  (cf. \cite{Han1997Newton,Li2018On,Lin2019Efficient}) of the solution to Problem (\ref{prob:vec_projection1}). Denote 
	\begin{align*}
		B=\left [
		\begin{array}{c}
			-F\\ 
			D \\
			b^\top 
		\end{array} \right]
		\in \mathbb{R}^{(2nd+1) \times nd} \ \ \mbox{and} \  \
		\ c = \left [ 
		\begin{array}{c}
			\textbf{0} \\
			\gamma
		\end{array}
		\right ]
		\in \mathbb{R}^{2nd+1}.
	\end{align*} 
	Problem (\ref{prob:vec_projection1}) then amounts to
	\begin{equation}\label{prob:vec_projection2}
		\min_{w} \Big\{\frac{1}{2} \|w-q\|^2 \mid Bw \leq c\Big\}.
	\end{equation} 
	The feasible set of Problem (\ref{prob:vec_projection2}) is defined by
	\begin{align*}
		\mathcal{C}:=\{w \in \mathbb{R}^{nd} \mid Bw \leq c\}.
	\end{align*}
	For given $q\in\mathbb{R}^{nd}$, we denote the unique optimal solution to Problem (\ref{prob:vec_projection2}) by $\Pi_{\mathcal{C}}(q)$. There exists a Lagrange multiplier $\mu\in\mathbb{R}_+^{2nd+1}$ satisfying the Karush-Kuhn-Tucker (KKT) conditions for Problem (\ref{prob:vec_projection2})
	\begin{equation}\label{equa:kkt}
		\left\{ 
		\begin{array}{l}
			\Pi_\mathcal{C}(q)-q+B^\top \mu =0,  \\ [5pt]
			B\Pi_\mathcal{C}(q)-c\leq 0,\\ [5pt]
			\mu^\top  (B\Pi_\mathcal{C}(q)-c) = 0.
		\end{array} \right.
	\end{equation} 
	Recall that the generalized HS-Jacobian for $\Pi_\mathcal{C}(\cdot)$ at any $q\in\mathbb{R}^{nd}$ as in \cite{Han1997Newton} is given by
	\begin{equation}\label{equa:hs_jacobi}
		\mathcal{N}_\mathcal{C}(q):=\big\{N\in\mathbb{R}^{nd \times nd}\mid N=I_{nd}-B_K^\top (B_K  B_K^\top )^{-1}B_K,K \in \mathcal{K}_{\mathcal{C}} (q) \big\},
	\end{equation}
	where
	\begin{align*}
		\mathcal{K}_{\mathcal{C}} (q):= \big\{K \subseteq \{1,\ldots,2nd+1\} \mid {\rm supp} (\mu) \subseteq K \subseteq \mathcal{I}(q),
		B_K \ \mbox{is of full row rank} \big\},
	\end{align*}
	here, $\mu$ is a Lagrange multiplier satisfying the KKT conditions (\ref{equa:kkt}), ${\rm supp}(\mu)=\{i\mid \mu_i\neq 0,i=1,\ldots,2nd+1\}$, and $\mathcal{I}(q) := \{ i \mid (B\Pi_\mathcal{C}(q))_i =c_i, i=1,2,\ldots,2nd+1\}$. To effectively implement the algorithm designed for Problem (\ref{mtlproblem1}), for any $q\in\mathbb{R}^{nd}$, we need to construct at least one computable element in $\mathcal{N}_\mathcal{C}(q)$ explicitly.
	Denote
	\begin{align}\label{equa:N0}
		N_0:=I_{nd}-B_{\mathcal{I}(q)}^\top (B_{\mathcal{I}(q)}B_{\mathcal{I}(q)}^\top )^{\dagger}B_{\mathcal{I}(q)}.
	\end{align}
	It follows from \cite[Lemma 2.1]{Han1997Newton} and \cite[Theorem 1 \& Proposition 1]{Li2020On} that $N_0 \in \mathcal{N}_\mathcal{C}(q)$. To simplify the procedure for computing a matrix $N_0$, we denote
	\begin{align}\label{alpha}
		\alpha_i=
		\left\{ 
		\begin{array}{ll}
			d, & {\rm rem}(i,d)=0,  \\ [5pt]
			{\rm rem}(i,d), & {\rm otherwise},
        \end{array} \right.
		\ i = 1,\ldots,nd
	\end{align} 
	and define some useful index sets by
	\begin{align}\label{idx}
		\mathcal{K}_{1} & :=\{ i\in \{1,\ldots,nd\} \mid j=\lambda(\alpha_i), (\Pi_\mathcal{C}(q))_j = 0, j\in\{1,\ldots,nd\}\}, \nonumber \\
		\mathcal{K}_{2} & :=  \{ i\in \{1,\ldots,nd\} \mid i=\lambda(\alpha_i) \}, \nonumber \\ 
		\mathcal{I}_{1} & := \{ i\in \{1,\ldots,nd\} \mid (F\Pi_\mathcal{C}(q))_i = 0\}, \nonumber\\
		\mathcal{I}_{2} & := \{ i\in \{1,\ldots,nd\} \mid (D\Pi_\mathcal{C}(q))_i = 0 \}, \\ 
		\mathcal{I}_{3} & := \{ i\in \{1,\ldots,nd\} \mid F_{ii}\neq 0, i\in\mathcal{I}_{1}\}, \nonumber\\	 
		\mathcal{I}_{4} & :=  \mathcal{I}_{2} \backslash (\mathcal{K}_1\cup \mathcal{K}_2),\nonumber\\
		\mathcal{K}_{3} & :=\{ i\in \{1,\ldots,nd\} \mid i=\lambda(\alpha_j), j\in \mathcal{I}_{4}\},\nonumber\\
		\mathcal{K}_{4}(j) & :=\{ i\in \{1,\ldots,nd\} \mid j=\lambda(\alpha_i),i\neq j \}, \forall j\in \mathcal{K}_{3}. \nonumber  
	\end{align}
	It is easy to see that $\mathcal{I}_{4}=\bigcup_{k\in \mathcal{K}_3}\mathcal{K}_4(k)$ and $\mathcal{I}_{3}$, $\mathcal{I}_{4}$ are pairwise disjoint sets. 
	
	With the above preparations, we are in a position to present the explicit expression of the matrix $N_0$ given in (\ref{equa:N0}) in the following theorem. Since its proof can be directly obtained from \cite[Theorem 2.1]{Lin2024An}, we omit them here for brevity.
	
	\begin{theorem}\label{TheoN0}
		Let $b, q\in \mathbb{R}^{nd}$ and $\gamma>0$ be given. Denote
		\begin{align*}
			\beta:=\hat{d}-\sum_{i\in\mathcal{K}_3}\frac{\lvert\mathcal{K}_4(i)\rvert}{1+\lvert\mathcal{K}_4(i)\rvert} \ \ {\rm with} \ \ \hat{d}=
			\left\{ 
			\begin{array}{ll}
				\lvert\mathcal{K}_2\setminus(\mathcal{K}_1\cap\mathcal{K}_2)\rvert, & \mathcal{K}_1\neq \emptyset,  \\ [5pt]
				d,& {\rm otherwise}. 
			\end{array} \right.
		\end{align*} 
		Then, the matrix $N_0$ given in (\ref{equa:N0}) admits the closed-form expression:
		\begin{align*} 
			N_0:=
			\left\{ 
			\begin{array}{ll}
				G-\displaystyle{ \frac{1}{\beta}}ff^\top , & b^\top \Pi_\mathcal{C}(q)= \gamma,  \\ [5pt]
				G,& b^\top \Pi_\mathcal{C}(q)\neq \gamma, 
			\end{array} \right.{\sc }
		\end{align*}
		where $G\in\mathbb{R}^{nd\times nd}$ and $f\in\mathbb{R}^{nd}$ are given by for $i, j=1,\dots,nd$,
		\begin{align*}
			G_{ij}=
			\left\{ 
			\begin{array}{lll}
				1,& i=j, \ j\in\mathcal{I}_3^c\setminus(\mathcal{K}_3\cup\mathcal{I}_{4}),  \\ [5pt]
				\displaystyle{\frac{{\rm sign}(q_iq_j)}{1+\mid\mathcal{K}_4(\lambda(\alpha_i))\mid}}, & i\in\mathcal{K}_3\cup\mathcal{I}_{4},\ j\in\{\lambda(\alpha_i)\}\cup\mathcal{K}_4(\lambda(\alpha_i)), \\[5pt]
				0,& {\rm otherwise}
			\end{array} \right. 
		\end{align*} 
		and
		\begin{align*}
			f_i=
			\left\{ 
			\begin{array}{lll}
				{\rm sign}(q_i),& i \in\mathcal{K}_2\setminus(\mathcal{K}_3\cup(\mathcal{K}_1\cap\mathcal{K}_2)),  \\ [5pt]
				\displaystyle{ \frac{{\rm sign}(q_i)}{1+\mid\mathcal{K}_4(\lambda(\alpha_i))\mid}}, & i \in\mathcal{K}_3\cup\mathcal{I}_{4},  \\ [5pt]
				0,& {\rm otherwise}.
			\end{array} \right. 
		\end{align*} 
	\end{theorem}

	\section{The adaptive sieving strategy}\label{sec:AS}
	In this section, we are going to develop an adaptive sieving (AS) strategy \cite{Lin2020Adaptive} to generate a solution path for multi-task Lasso problems. The main idea of this strategy is to reduce the number of variables by sieving out a large number of inactive features, thereby reducing the dimension of multi-task Lasso problems and obtaining its solution paths with a sequence of given positive parameters. We denote our proposed algorithm by AS-{\sc Ssnpal}, since a semismooth Newton proximal augmented Lagrangian algorithm is utilized to solve inner problems of the AS strategy. 
	
	Let $m=\sum_{i=1}^n m_i$ and $y=[\hat{y}^1;\ldots;\hat{y}^n]\in\mathbb{R}^{m}$. Define the linear operator $\mathcal{X}:\mathbb{R}^{d\times n}\rightarrow \mathbb{R}^{m}$ by
	\begin{align*}
		\mathcal{X}(W):=[X^1\hat{w}^1;X^2\hat{w}^2;\ldots;X^n\hat{w}^n], \, \ \forall \ W\in\mathbb{R}^{d\times n}.
	\end{align*}
	Denote the feasible set of Problem (\ref{mtlproblem1}) by
	\begin{equation*}\label{Omega}
		\Omega_\gamma:=\{W\in\mathbb{R}^{d\times n} \mid \|W\|_{1,\infty} \leq \gamma\}.
	\end{equation*}
	Then, Problem (\ref{mtlproblem1}) can be reformulated as
	\begin{equation}\label{problemP}\tag{${\rm P}_\gamma$}
		\begin{split}
			\min_{W} & \ \   \frac{1}{2} \|y-\mathcal{X}(W)\|^2+\chi_{\Omega_\gamma}(W),
		\end{split}
	\end{equation} 
	where $\chi_{\Omega_\gamma}:\mathbb{R}^{d\times n}\rightarrow\mathbb{R}\cup \{+\infty\}$ is the indicator function of the convex closed set $\Omega_\gamma$. Denote the optimal solution set of Problem (\ref{problemP}) by $\Theta_{\gamma}$ and the associated proximal residual function $R_\gamma:\mathbb{R}^{d \times n}\rightarrow \mathbb{R}^{d \times n}$ by
	\begin{equation}\label{equa:R_gamma}
		R_\gamma(W):=W-{\rm Prox}_{\chi_{\Omega_{\gamma}}}(W-\mathcal{X}^*(\mathcal{X}W-y)), \ \forall \ W\in \mathbb{R}^{d \times n}.
	\end{equation}
	Obviously, the set $\Theta_{\gamma}$ is nonempty and compact. It then follows from the KKT condition of (\ref{problemP}) that $W^*\in\Theta_{\gamma}$ if and only if $R_\gamma(W^*)=0$. 
	
	\begin{algorithm}
		\caption{\small {\bf An adaptive sieving strategy for Problem (\ref{problemP})}}\label{algor:AS}
		\hspace*{0.01in} \raggedright {\bf Input:} a sequence of parameter: $0<\gamma_0<\gamma_1<\ldots<\gamma_p$, and tolerance $\epsilon>0$. \\
		\hspace*{0.01in} \raggedright {\bf Output:} a solution path: $W^*(\gamma_0),W^*(\gamma_1), \ldots, W^*(\gamma_p)$. \\
		\begin{algorithmic}[1]
			\STATE \textbf{Initialization}: for $\gamma_0>0$, find
			\begin{equation}\label{equa:Delta_0}
				W^*(\gamma_0)\in \underset{W\in\mathbb{R}^{d \times n}} {\rm arg\ min}\ \Big\{ \frac{1}{2} \|y-\mathcal{X}(W)\|^2+\chi_{\Omega_{\gamma_0}}(W)-\langle \Delta_0,W\rangle \Big\},
			\end{equation}
			where $\Delta_0\in\mathbb{R}^{d \times n}$ is an error matrix such that $\|\Delta_0\|_F\leq\epsilon/\sqrt{2}$. Let
			\begin{equation*}
				I^*(\gamma_0)=\{(i,j)\mid (W^*(\gamma_0))_{ij}\neq0,i=1,\ldots,d;\ j=1,\ldots,n \}.
			\end{equation*}
			
			\FOR
			{$k=1, 2, \ldots, p$}
			\STATE Let $I^0(\gamma_k)=I^*(\gamma_{k-1})$.
			Solve
			\begin{equation*}\label{prob:Delta_0}
				W^0(\gamma_k)\in \underset{W\in\mathbb{R}^{d \times n}} {\rm arg\ min}\ \Big\{ \frac{1}{2} \|y-\mathcal{X}(W)\|^2+\chi_{\Omega_{\gamma_k}}(W)-\langle \Delta_k^0,W\rangle \mid W_{\bar{I}^0(\gamma_k)}=0 \Big\},
			\end{equation*}
			where $\bar{I}^0(\gamma_k)$ is the complement of $I^0(\gamma_k)$, i.e.,
			\begin{equation*}
				\bar{I}^0(\gamma_k)=\{(i,j)\mid (W^*(\gamma_k))_{ij}=0,i=1,\ldots,d;\ j=1,\ldots,n \},
			\end{equation*}
			and $\Delta_k^0\in\mathbb{R}^{d \times n}$ is an error matrix such that $\|\Delta_k^0\|_F\leq\epsilon/\sqrt{2}$, $(\Delta_k^0)_{\bar{I}^0(\gamma_k)}=0$.
			\STATE Compute $R_{\gamma_k}(W^0(\gamma_k))$ and set $l=0$.
			\WHILE {$R_{\gamma_k}(W^l(\gamma_k))>\epsilon$}
			\STATE Create $J^{l+1}(\gamma_k)$:
			\begin{equation}\label{equa:J}
				J^{l+1}(\gamma_k)= \Big\{(i,j)\in	\bar{I}^l(\gamma_k)\mid (\mathcal{X}^*(y-\mathcal{X}W^l(\gamma_k)))_{ij}\notin (\partial \chi_{\Omega_{\gamma_k}}(W^l(\gamma_k))+\frac{\epsilon}{\sqrt{2\lvert \bar{I}^l(\gamma_k)\rvert}}\mathbb{B}_\infty)_{ij}   \Big\},
			\end{equation}
			where $\bar{I}^l(\gamma_k)$ is the complement of $I^l(\gamma_k)$,  $\mathbb{B}_\infty$ denotes the infinity norm unit ball, and $(\mathcal{P})_{ij}$ is the projection of the set $\mathcal{P}$ onto the $(i,j)$-th dimension.
			\STATE Update $I^{l+1}(\gamma_k)\leftarrow I^{l}(\gamma_k)\cup J^{l+1}(\gamma_k)$.
			\STATE Solve the following problem
			\begin{equation}\label{equa:Delta_l1}
				W^{l+1}(\gamma_k)\in \underset{W\in\mathbb{R}^{d \times n}} {\rm arg\ min}\ \Big\{ \frac{1}{2} \|y-\mathcal{X}(W)\|^2+\chi_{\Omega_{\gamma_k}}(W)-\langle \Delta_k^{l+1},W\rangle \mid W_{\bar{I}^{l+1}(\gamma_k)}=0 \Big\},
			\end{equation}
			where $\Delta_k^{l+1}\in\mathbb{R}^{d \times n}$ is an error matrix such that $\|\Delta_k^{l+1}\|_F\leq\epsilon/\sqrt{2}$, $(\Delta_k^{l+1})_{\bar{I}^{l+1}(\gamma_k)}=0$.
			\STATE Compute $R_{\gamma_k}(W^{l+1}(\gamma_k))$ and set $l\leftarrow l+1$.
			\ENDWHILE
			\STATE Set $W^*(\gamma_k)=W^l(\gamma_k)$ and $I^*(\gamma_k)=I^l(\gamma_k)$.
			\ENDFOR
		\end{algorithmic}
	\end{algorithm}
	
	In Algorithm \ref{algor:AS}, it is worthwhile to mention that the existence of error vectors $\Delta^0$ and $\{\Delta_k^{l+1}\}$ in (\ref{equa:Delta_0}) and (\ref{equa:Delta_l1}) means the involved minimization problems cannot be solved exactly. In fact, they are not given in prior but incurred automatically when the original problems (with $\Delta^0=0$ and $\{\Delta_k^{l+1}\}=0$) are solved inexactly. As an illustration, the following proposition will show how the error $\Delta^0$ is generated.
	
	
	
	\begin{proposition}
		Given the parameter $\gamma_0>0$. Denote the Lipschitz constant of $\mathcal{X}^*(\mathcal{X}W-y)$ by $L$, i.e., the largest eigenvalue of $\mathcal{X}^*\mathcal{X}$. Then the optimal solution $W^*(\gamma_0)$ in (\ref{equa:Delta_0}) can be computed by
		\begin{equation}\label{equa:W_gamma0}
			W^*(\gamma_0): = {\rm Prox}_{\chi_{\Omega_{\gamma_0}}}(\widehat{W}(\gamma_0)-\mathcal{X}^*(\mathcal{X}\widehat{W}(\gamma_0)-y)),
		\end{equation}
		where $\widehat{W}(\gamma_0)$ is an approximate solution to 
		\begin{equation}\label{prob:P_gamma0}
			\begin{split}
				\min_{W} & \ \   \frac{1}{2} \|y-\mathcal{X}(W)\|^2+\chi_{\Omega_{\gamma_0}}(W)
			\end{split}
		\end{equation}
		such that 
		\begin{equation}\label{equa:R_X}
			\|R_{\gamma_0}(\widehat{W}(\gamma_0))\|\le \frac{\epsilon}{\sqrt{2}(1+L)}.
		\end{equation}
	\end{proposition}
	\begin{proof}
		For any $i=1,2,\ldots$, let {$W^i$} be a sequence that converges to the optimal solution of Problem (\ref{prob:P_gamma0}). Denote $\Delta^i:=R_{\gamma_0}(W^i)+\mathcal{X}^*\mathcal{X}({\rm Prox}_{\chi_{\Omega_{\gamma_0}}}(W^i-\mathcal{X}^*(\mathcal{X}W^i-y))-W^i)$. Then, from \cite[Lemma 4.5]{Du2015An}, it follows that $\lim_{i\rightarrow \infty}\|\Delta^i\|=0$, which means the existence of $\widehat{W}(\gamma_0)$ in (\ref{equa:R_X}).
		
		Next, we intend to demonstrate why $W^*(\gamma_0)$ can be computed by (\ref{equa:W_gamma0}). By the definition of $R_\gamma$, one has
		\begin{equation*}
			R_{\gamma_0}(\widehat{W}(\gamma_0))=\widehat{W}(\gamma_0)-{\rm Prox}_{\chi_{\Omega_{\gamma_0}}}(\widehat{W}(\gamma_0)-\mathcal{X}^*(\mathcal{X}\widehat{W}(\gamma_0)-y))=\widehat{W}(\gamma_0)-W^*(\gamma_0). 
		\end{equation*}
		Combining this with (\ref{equa:W_gamma0}), we obtain
		\begin{equation*}
			R_{\gamma_0}(\widehat{W}(\gamma_0))-\mathcal{X}^*(\mathcal{X}\widehat{W}(\gamma_0)-y)\in \partial \chi_{\Omega_{\gamma_0}}(W^*(\gamma_0)). 
		\end{equation*}
		If we choose $\Delta^0:=R_{\gamma_0}(\widehat{W}(\gamma_0))+\mathcal{X}^*\mathcal{X}(W^*(\gamma_0)-\widehat{W}(\gamma_0))$, it then holds 
		\begin{equation*}
			\Delta^0\in \mathcal{X}^*(\mathcal{X}W^*(\gamma_0)-y)+\partial \chi_{\Omega_{\gamma_0}}(W^*(\gamma_0)),
		\end{equation*} 
		which implies that $W^*(\gamma_0)$ is an optimal solution of 
		\begin{equation*}
			\min_{W\in \mathbb{R}^{d\times n}}\ \Big\{ \frac{1}{2} \|y-\mathcal{X}(W)\|^2+\chi_{\Omega_{\gamma_0}}(W)-\langle \Delta_0,W\rangle \Big\}
		\end{equation*}
		with the above $\Delta^0$. Moreover, we have
		\begin{align*}
			\epsilon/\sqrt{2} \ge\|\Delta^0\|_F=&\|R_{\gamma_0}(\widehat{W}(\gamma_0))+\mathcal{X}^*\mathcal{X}(W^*(\gamma_0)-\widehat{W}(\gamma_0))\|_F \\
			\le& \|R_{\gamma_0}(\widehat{W}(\gamma_0))\|_F+\|\mathcal{X}^*\mathcal{X}(W^*(\gamma_0)-\widehat{W}(\gamma_0))\|_F\\
			\le&\|R_{\gamma_0}(\widehat{W}(\gamma_0))\|_F+L\|W^*(\gamma_0)-\widehat{W}(\gamma_0)\|_F\\
			=&(1+L)\|R_{\gamma_0}(\widehat{W}(\gamma_0))\|_F,
		\end{align*}  
		where the first inequality follows from the triangle inequality and the second inequality uses the Lipschitz property of  $\mathcal{X}^*(\mathcal{X}W-y)$. Since $(1+L)\|R_{\gamma_0}(\widehat{W}(\gamma_0))\|_F$ is the least upper bound of $\|\Delta^0\|_F$, we have 
		\begin{equation*}
			(1+L)\|R_{\gamma_0}(\widehat{W}(\gamma_0))\|_F\le \epsilon/\sqrt{2}
		\end{equation*}
		which leads to the inequality (\ref{equa:R_X}).
	\end{proof}

	Now, we shall establish the convergence results of Algorithm \ref{algor:AS} in the following theorem.
	\begin{theorem}
		The solution path $(W^*(\gamma_0),W^*(\gamma_1), \ldots, W^*(\gamma_p))$ generated by Algorithm \ref{algor:AS} are approximate optimal solutions to a sequence problem $({\rm P}_{\gamma_0},{\rm P}_{\gamma_1}, \ldots, {\rm P}_{\gamma_p})$, that is, 
		\begin{equation*}
			\|R_{\gamma_k}(W^*(\gamma_k))\|_F\le \epsilon, \ k = 0,1,\ldots,p.
		\end{equation*}
	\end{theorem}
	\begin{proof}
		For any given $k=0,1,\ldots,p$, we first prove that for some $l\ge 0$, $J^{l+1}(\gamma_k)$ given in (\ref{equa:J}) is a nonempty index set as long as $\|R_{\gamma_k}(W^l(\gamma_k))\|_F> \epsilon$. To reach a contradiction, we assume that there exists $k\in\{0,1,\ldots,p\}$ such that if $\|R_{\gamma_k}(W^l(\gamma_k))\|_F>\epsilon$, the index set $J^{l+1}(\gamma_k)$ is empty, which yields that for any $(i,j)\in \bar{I}^l(\gamma_k)$,
		\begin{equation*}
			(\mathcal{X}^*(y-\mathcal{X}W^l(\gamma_k)))_{ij}\in (\partial \chi_{\Omega_{\gamma_k}}(W^l(\gamma_k))+\frac{\epsilon}{\sqrt{2\lvert \bar{I}^l(\gamma_k)\rvert}}\mathbb{B}_\infty)_{ij}.
		\end{equation*}
		Then, there exists a matrix $\widehat{\Delta}^l_k\in\mathbb{R}^{d\times n}$ with $(\widehat{\Delta}^l_k)_{I^l(\gamma_k)}=0$ and $\|\widehat{\Delta}^l_k\|_\infty\le \frac{\epsilon}{\sqrt{2\lvert \bar{I}^l(\gamma_k)\rvert}}$ such that
		\begin{equation}\label{equa:hat_Delta}
			(\mathcal{X}^*(y-\mathcal{X}W^l(\gamma_k))+\widehat{\Delta}^l_k)_{ij}\in (\partial \chi_{\Omega_{\gamma_k}}(W^l(\gamma_k)))_{ij}, \ \forall \ (i,j)\in \bar{I}^l(\gamma_k),
		\end{equation}
		which implies
		\begin{equation*}
			(W^l(\gamma_k))_{\bar{I}^l(\gamma_k)} = ({\rm Prox}_{\chi_{\Omega_{\gamma_k}}}(W^l(\gamma_k)+\mathcal{X}^*(y-\mathcal{X}W^l(\gamma_k))+\widehat{\Delta}^l_k))_{\bar{I}^l(\gamma_k)}.
		\end{equation*}
		Thus, we have 
		\begin{equation}\label{equa:R1}
			\begin{split}
				&\|(W^l(\gamma_k)- {\rm Prox}_{\chi_{\Omega_{\gamma_k}}}(W^l(\gamma_k)+\mathcal{X}^*(y-\mathcal{X}W^l(\gamma_k))_{\bar{I}^l(\gamma_k)}\|_F^2\\
				=&\|({\rm Prox}_{\chi_{\Omega_{\gamma_k}}}(W^l(\gamma_k)+\mathcal{X}^*(y-\mathcal{X}W^l(\gamma_k))+\widehat{\Delta}^l_k)-{\rm Prox}_{\chi_{\Omega_{\gamma_k}}}(W^l(\gamma_k)+\mathcal{X}^*(y-\mathcal{X}W^l(\gamma_k))_{\bar{I}^l(\gamma_k)}\|_F^2\\
				\le&\|(\widehat{\Delta}^l_k))_{\bar{I}^l(\gamma_k)}\|_F^2
				\le\big(\frac{\epsilon}{\sqrt{2\lvert \bar{I}^l(\gamma_k)\rvert}}\big)^2\lvert \bar{I}^l(\gamma_k)\rvert\\
				=& \epsilon^2/2,
			\end{split}
		\end{equation}
		where the first inequality follows from the fact that the proximal mapping is globally Lipschitz continuous with modulus $1$ and the second inequality is due to $\|\widehat{\Delta}^l_k\|_\infty\le \epsilon/\sqrt{2\lvert \bar{I}^l(\gamma_k)\rvert}$. Recall that $W^l(\gamma_k)$ is an optimal solution to the minimization problem
		\begin{equation}\label{prob:W_l}
			\underset{W\in\mathbb{R}^{d \times n}} {\rm min}\ \Big\{ \frac{1}{2} \|y-\mathcal{X}(W)\|^2+\chi_{\Omega_{\gamma_k}}(W)-\langle \Delta_k^{l},W\rangle \mid W_{\bar{I}^{l}(\gamma_k)}=0 \Big\},
		\end{equation}
		where $\Delta_k^{l}\in\mathbb{R}^{d \times n}$ is an error matrix such that $\|\Delta_k^{l}\|_F\leq\epsilon/\sqrt{2}$, $(\Delta_k^{l})_{\bar{I}^{l}(\gamma_k)}=0$. Denote $\Lambda\in \mathbb{R}^{d \times n}$ with $\Lambda_{I^{l}(\gamma_k)}=0$ by the Lagrange multiplier associated with the constraint of the above problem. Then the KKT conditions for Problem (\ref{prob:W_l}) are as follows
		\begin{equation*}
			\left\{ 
			\begin{array}{l}
				0\in (\mathcal{X}^*(\mathcal{X}W^l(\gamma_k)-y)+\partial \chi_{\Omega_{\gamma_k}}(W^l(\gamma_k))-\Delta_k^{l})_{I^{l}(\gamma_k)}, \\ 
				0\in (\mathcal{X}^*(\mathcal{X}W^l(\gamma_k)-y)+\partial \chi_{\Omega_{\gamma_k}}(W^l(\gamma_k))-\Delta_k^{l}-\Lambda)_{\bar{I}^{l}(\gamma_k)}, \\ 
				(W^l(\gamma_k))_{\bar{I}^{l}(\gamma_k)}=0,
			\end{array} \right.  	
		\end{equation*}
		which yields
		\begin{equation*}
			(W^l(\gamma_k))_{I^l(\gamma_k)} = ({\rm Prox}_{\chi_{\Omega_{\gamma_k}}}(W^l(\gamma_k)+\mathcal{X}^*(y-\mathcal{X}W^l(\gamma_k))+{\Delta}^l_k))_{I^l(\gamma_k)},
		\end{equation*}
		and hence
		\begin{equation*}\label{equa:R2}
			\begin{split}
				&\|(W^l(\gamma_k)- {\rm Prox}_{\chi_{\Omega_{\gamma_k}}}(W^l(\gamma_k)+\mathcal{X}^*(y-\mathcal{X}W^l(\gamma_k)))_{I^l(\gamma_k)}\|_F^2\\
				=&\|({\rm Prox}_{\chi_{\Omega_{\gamma_k}}}(W^l(\gamma_k)+\mathcal{X}^*(y-\mathcal{X}W^l(\gamma_k))+\Delta^l_k)-{\rm Prox}_{\chi_{\Omega_{\gamma_k}}}(W^l(\gamma_k)+\mathcal{X}^*(y-\mathcal{X}W^l(\gamma_k)))_{I^l(\gamma_k)}\|_F^2\\
				\le&\|({\Delta}^l_k))_{I^l(\gamma_k)}\|_F^2\\
				\le& \epsilon^2/2,
			\end{split}
		\end{equation*}
		which, combining with (\ref{equa:R1}), due to the definition of $R_{\gamma}$ given in (\ref{equa:R_gamma}), implies
		\begin{align*}
			&\|R_{\gamma_k}(W^l(\gamma_k))\|_F^2\\
			=& \|(W^l(\gamma_k)- {\rm Prox}_{\chi_{\Omega_{\gamma_k}}}(W^l(\gamma_k)+\mathcal{X}^*(y-\mathcal{X}W^l(\gamma_k)))_{\bar{I}^l(\gamma_k)}\|_F^2\\
			&+\|(W^l(\gamma_k)- {\rm Prox}_{\chi_{\Omega_{\gamma_k}}}(W^l(\gamma_k)+\mathcal{X}^*(y-\mathcal{X}W^l(\gamma_k)))_{I^l(\gamma_k)}\|_F^2\\
			\le& \epsilon^2/2+\epsilon^2/2 = \epsilon^2.
		\end{align*}
		Obviously, this is a contradiction to $\|R_{\gamma_k}(W^l(\gamma_k))\|_F>\epsilon$. Therefore, we conclude that for some $l\ge0$, $J^{l+1}(\gamma_k)\neq \emptyset$ as long as $\|R_{\gamma_k}(W^l(\gamma_k))\|_F> \epsilon$. Since the total number of components of $W$ is finite, Algorithm \ref{algor:AS} will terminate after a finite number of iterations. This completes the proof.
	\end{proof}

	\section{A semismooth Newton proximal augmented Lagrangian algorithm}\label{sec:Ssnpal}
	
	In this section, we shall employ the semismooth Newton proximal augmented Lagrangian algorithm to solve the minimization problems in Step 1 to Step 8 of Algorithm \ref{algor:AS}. Specifically, the proximal augmented Lagrangian (PAL) algorithm is proposed for solving the equivalent form of the original problem and the semismooth Newton ({\sc Ssn}) algorithm is applied to solve the subproblems of the PAL algorithm. 
	
	For any $\gamma\in\{\gamma_i,\ i=0,1,\ldots,p\}$, the original problems in Step 1 to Step 8 of Algorithm \ref{algor:AS} are as follows
	\begin{equation}\label{prob:P_As}
		\underset{W\in\mathbb{R}^{d \times n}} {\rm min}\ \Big\{ \frac{1}{2} \|y-\mathcal{X}(W)\|^2+\chi_{\Omega_{\gamma}}(W) \mid W\in \mathcal{S}_{\gamma} \Big\},
	\end{equation} 
	where $\mathcal{S}_{\gamma}:=\{W\in\mathbb{R}^{d \times n}\mid W_{ij}=0, (i,j)\in\bar{I}(\gamma)\}$. Note that if the PAL algorithm is directly applied to solve Problem (\ref{prob:P_As}), the resulting inner problem is nonsmooth composite minimization due to the nonsmoothness of indicator function $\chi_{\Omega_\gamma}$, which means that the semismooth Newton algorithm cannot solve the subproblem efficiently. Therefore, we introduce an auxiliary variable $Z\in\mathbb{R}^{d\times n}$ and reformulate Problem (\ref{prob:P_As}) as
	\begin{equation}\tag{${\rm}P_\gamma^\prime$}\label{prob:problemp'}
		\min_{W,Z} \Big\{\frac{1}{2} \|y-\mathcal{X}(W)\|^2+\chi_{\mathcal{S}_{\gamma}}(W)+\chi_{\Omega_\gamma}(Z) \mid W-Z=0\Big\}.
	\end{equation}
	
	The Lagrangian function of Problem (\ref{prob:problemp'}) is given by
	\begin{equation*}
		\ell(W,Z;U):=\frac{1}{2} \|y-\mathcal{X}(W)\|^2+\chi_{\mathcal{S}_{\gamma}}(W)+\chi_{\Omega_\gamma}(Z)+\langle U,W-Z\rangle.
	\end{equation*}
	Then the corresponding KKT optimality conditions for Problem (\ref{prob:problemp'}) are 
	\begin{equation}\label{equa:kkt_P'}
		0\in\mathcal{X}^*(\mathcal{X}(W)-y)+U+\partial\chi_{\mathcal{S}_{\gamma}}(W), \ \ U\in\partial \chi_{\Omega_\gamma}(Z), \ \ W-Z=0.
	\end{equation}
	Given $\sigma>0$, the augmented Lagrangian function of Problem (\ref{prob:problemp'}) is given by
	\begin{equation*}
		\mathcal{L}_\sigma(W,Z;U):=\ell(W,Z;U)+\frac{\sigma}{2}\|W-Z\|_F^2.
	\end{equation*}
	In the following, we apply the proximal augmented Lagrangian algorithm originally proposed in \cite{Rockafellar1976Augmented} to solve Problem (\ref{prob:problemp'}), whose main idea is to achieve the desired strong convexity of the involved function by adding a proximal term. Given $\sigma_k>0, \ U^k\in\mathbb{R}^{d\times n}$ and $W^k\in\mathbb{R}^{d\times n}$, we consider the proximal minimization subproblem as follows:
	\begin{equation}\label{prob:proxiprob}
		\min_{W\in\mathbb{R}^{d\times n}}  \Big\{\varPhi_k(W):=\inf_Z \mathcal{L}_{\sigma_k}(W,Z;U^k)+\frac{1}{2\sigma_k}\|W-W^k\|_F^2   \Big\},
	\end{equation}
	where $\varPhi_k(W)$ can be explicitly computed by
	\begin{equation*}
		\begin{split} 
			\varPhi_k(W)
			=& \frac{1}{2} \|y-\mathcal{X}(W)\|^2+\chi_{\mathcal{S}_{\gamma}}(W)-\frac{1}{2\sigma_k}\|U^k\|_F^2+\sigma_k\inf_Z \Big\{\chi_{\Omega_\gamma}(Z)+\frac{1}{2}\|Z-(W+U^k/\sigma_k)\|_F^2\Big\}\\
			&+\frac{1}{2\sigma_k}\|W-W^k\|_F^2.
		\end{split}
	\end{equation*}
	
	For the purpose of implementing the PAL algorithm efficiently, we denote some notations to reformulate Problem (\ref{prob:proxiprob}) into a small size minimization problem. For any $S\in \mathbb{R}^{d\times n}$, define a linear operator $\mathcal{Q}:\mathbb{R}^{d\times n}\rightarrow \mathbb{R}^t$ as follows: let $\bar{s}:=\mathcal{Q}(S)$ be the subvector of ${\rm Vec}(S)$ with the elements $S_{ij},\ (i,j)\in I(\gamma)$ being extracted, where $t=|I(\gamma)|$. The generalized inverse of $\mathcal{Q}$ is denoted by $\mathcal{Q}^\dagger:\mathbb{R}^t\rightarrow \mathbb{R}^{d\times n}$ such that $S=\mathcal{Q}^\dagger(\bar{s})$. Let $\bar{w}:=\mathcal{Q}(W)$, $\bar{z}:=\mathcal{Q}(Z)$, $\bar{u}:=\mathcal{Q}(U)$, and $\bar{b}:=\mathcal{Q}(T)$, where the matrix $T$ is defined in (\ref{equa:T}). Define an index set by
	\begin{equation*}
		\mathcal{J}(\gamma):=\{k\in\{1,2,\ldots,nd\}\mid k=(j-1)d+i, \ (i,j)\in I(\gamma)\}.
	\end{equation*}
	Then solving Problem (\ref{prob:proxiprob}) is equivalent to solving the following problem
	\begin{equation}\label{prob:phi}
		\begin{split} 
			\min_{\bar{w}\in\mathbb{R}^{t}} \big\{\varphi_k(\bar{w})
			:=& \frac{1}{2} \|y-X_{[m]\mathcal{J}(\gamma)}\bar{w}\|^2-\frac{1}{2\sigma_k}\|\bar{u}^k\|^2+\sigma_k\inf_{\bar{z}} \Big\{\chi_{\bar{\mathcal{C}}}(\bar{z})+\frac{1}{2}\|\bar{z}-(\bar{w}+\bar{u}^k/\sigma_k)\|^2\Big\}\\
			&+\frac{1}{2\sigma_k}\|\bar{w}-\bar{w}^k\|^2\big\},
		\end{split}
	\end{equation}
	where
	\begin{align*}
		\bar{\mathcal{C}}=\{w\in\mathbb{R}^{t} \mid \bar{B}w\le c\}
		\ \mbox{with} \
		\bar{B}=\left [
		\begin{array}{c}
			-F_{[nd]\mathcal{J}(\gamma)}\\ 
			D_{[nd]\mathcal{J}(\gamma)} \\
			\bar{b}^\top 
		\end{array} \right]
		\in \mathbb{R}^{(2nd+1) \times t},
	\end{align*}
	and $X\in \mathbb{R}^{m\times nd}$ is the matrix representation of $\mathcal{X}$ with respect to the standard basis of $\mathbb{R}^{d\times n}$ and $\mathbb{R}^{m}$. It is not difficult to see that $X$ is a diagonal block matrix with $X={\rm Diag}(X^1,X^2,\dots,X^n)$.

	\subsection{A proximal augmented Lagrangian algorithm}
	The framework of the proximal augmented Lagrangian algorithm for solving Problem (\ref{prob:problemp'}) is summarized as follows.
	
	\begin{algorithm}
		\caption{\small {\bf A proximal augmented Lagrangian algorithm for solving Problem (\ref{prob:problemp'})}} \label{algo:palm}
		\hspace*{0.01in} \raggedright {\bf Input:} 
		Choose an initial point $(W^0,Z^0,U^0)\in \mathbb{R}^{d\times n} \times \mathbb{R}^{d\times n}\times \mathbb{R}^{d\times n}$ and a parameter $\sigma_0>0$. \\
		\hspace*{0.01in}  \raggedright {\bf Output:} 
		An approximate optimal solution $(\hat{W},\hat{Z},\hat{U})$.\\
		\begin{algorithmic}[1]		
			\WHILE{Stopping criteria are not met}
			\STATE  Let $\bar{w}^0=\mathcal{Q}(W^0)$ and $\bar{u}^0=\mathcal{Q}(U^0)$. Iterate the following steps with $k=0,1,\ldots$
			\STATE  Approximately compute
			\begin{equation}\label{equa:wk1}
				\bar{w}^{k+1} \approx \underset{\bar{w}\in\mathbb{R}^{d\times n}}{\rm arg \ min} \big\{\varphi_k(\bar{w})\big\}
			\end{equation}
			to satisfy conditions (\ref{stopA1}) and (\ref{stopA2}), where $\varphi_k(\cdot)$ is defined in (\ref{prob:phi}). Extend $\bar{w}^{k+1}$ to $W^{k+1}$ by $W^{k+1}=\mathcal{Q}^\dagger(\bar{w}^{k+1})$.

			\STATE 
			Compute 
			\begin{equation*}
				Z^{k+1}=\Pi_{\Omega_\gamma}(W^{k+1}+U^k/\sigma_k), \ 
				\bar{z}^{k+1}=\mathcal{Q}(Z^{k+1}).
			\end{equation*}

			\STATE 
			Compute 
			\begin{equation*}
				U^{k+1} = U^k+\sigma_k(W^{k+1}-Z^{k+1}), \ 
				\bar{u}^{k+1}=\mathcal{Q}(U^{k+1}).
			\end{equation*}

			\STATE 
			Update $\sigma_{k+1} \uparrow \sigma_\infty\leq \infty$, $\hat{W}=W^{k+1}$, $\hat{Z}=Z^{k+1}$, and $\hat{U}=U^{k+1}$.
			
			\STATE
			$k:=k+1$.
			\ENDWHILE
		\end{algorithmic}
	\end{algorithm}
	In Algorithm \ref{algo:palm}, we obtain the approximate solution $\bar{w}^{k+1}$ to subproblem (\ref{equa:wk1}) with the following stopping criteria
	\begin{align}
		\|\nabla\varphi_k(\bar{w}^{k+1})\| & \le \frac{\varepsilon_k, \ }{\sigma_k},\ 0\le \varepsilon_k, \ \sum_{k=0}^{\infty} \varepsilon_k <\infty, \tag{A1}\label{stopA1}\\
		\|\nabla\varphi_k(\bar{w}^{k+1})\| & \le \frac{\vartheta_k}{\sigma_k}\|(\bar{w}^{k+1},\bar{z}^{k+1})-(\bar{w}^k,\bar{z}^k)\|+\|U^{k+1}-U^k\|_F,\  0\le \vartheta_k<1, \ \sum_{k=0}^{\infty} \vartheta_k <\infty. \tag{A2}\label{stopA2}
	\end{align} 
	
	Actually, the PAL algorithm has been studied in \cite[section 5]{Rockafellar1976Augmented}, we refer the reader to \cite{Rockafellar1976Augmented} for its detailed discussions. The attractive results on global and local convergence of Algorithm \ref{algo:palm} can be obtained directly from \cite{Luque1984Asymptotic,Rockafellar1976Augmented,Rockafellar1976Monotone}. Before we state them in the following theorem without any proof for brevity, we begin with analyzing the error bound condition for the involved function.
	
	Define the maximal monotone operator $\mathcal{T}_\ell$ as 
	\begin{equation*}
		\begin{split}
			\mathcal{T}_\ell(W,Z,U):=&\big\{(W^\prime,Z^\prime,U^\prime)\mid (W^\prime,Z^\prime,-U^\prime)\in\partial \ell(W,Z,U)\big\}\\
			=&\big\{(W^\prime,Z^\prime,U^\prime)\mid W^\prime\in\mathcal{X}^*(\mathcal{X}(W)-y)+\partial\chi_{\mathcal{S}_\gamma}(W)+U, \\
            &\ Z^\prime\in \partial \chi_{\Omega_\gamma}(Z)-U,\ U^\prime= W-Z\big\},
		\end{split}
	\end{equation*} 
	whose corresponding inverse operator is given by
	\begin{equation*}
		\begin{split}
			\mathcal{T}_\ell^{-1}(W^\prime,Z^\prime,U^\prime):=&\underset{W,Z,U}{\rm arg \  minimax}\big\{\ell(W,Z,U)-\langle W^\prime,W\rangle-\langle Z^\prime,Z\rangle+\langle U^\prime,U\rangle\big\}\\
			=&\big\{(W,Z,U)\mid (W^\prime,Z^\prime,-U^\prime)\in\partial \ell(W,Z,U)\big\}.
		\end{split}
	\end{equation*} 
	Since the set $\Omega_\gamma$ is polyhedral, $\partial \chi_{\Omega_\gamma}$ is a polyhedral multifunction (cf. \cite{Klatte2002Nonsmooth}), showing that $\mathcal{T}_\ell$ is also a polyhedral multifunction. It then follows from \cite[Proposition 12.30]{Rockafellar1998Variational} that $\mathcal{T}_\ell$ satisfies the error bound condition. Suppose that $r$ is a real number with $r>\sum_{k=0}^{\infty}\varepsilon_k$, where $\varepsilon_k$ is the summable sequence defined in (\ref{stopA1}). For given a positive modulus $\kappa>0$, if ${\rm dist}((W,Z,U),\mathcal{T}_\ell^{-1}(0))\le r$, then 
	\begin{equation*}
		{\rm dist}((W,Z,U),\mathcal{T}_\ell^{-1}(0))\le \kappa{\rm dist}(0,\mathcal{T}_\ell((W,Z,U)).
	\end{equation*} 
	
	Based on the above analysis, one can establish convergence results of Algorithm \ref{algo:palm} in the next theorem.
	
	\begin{theorem}
		(1) Suppose that ${(W^k,Z^k,U^k)}$ is the sequence generated by Algorithm \ref{algo:palm} with the stopping criterion \ref{stopA1}. Then the sequence ${(W^k,Z^k,U^k)}$ is bounded, $(W^k,Z^k)$ converges to an optimal solution Problem (\ref{prob:problemp'}), and $U^k$ converges to an optimal solution to the dual problem of Problem (\ref{prob:problemp'}).
		
		(2) Let ${(W^k,Z^k,U^k)}$ be the sequence generated by Algorithm \ref{algo:palm} with the stopping criteria (\ref{stopA1}) and (\ref{stopA2}). Then, for all $k>0$, it holds that 
		\begin{equation*}
			{\rm dist}((W^{k+1},Z^{k+1},U^{k+1}),\mathcal{T}_{\ell}^{-1}(0))\le \mu_k{\rm dist}((W^k,Z^k,U^k),\mathcal{T}_{\ell}^{-1}(0)),
		\end{equation*} 
		where 
		\begin{equation*}
			\mu_k=\frac{1}{1-\vartheta_k}\Big(\vartheta_k+\frac{\kappa(1+\vartheta_k)}{\sqrt{\sigma_k^2+\kappa^2}}\Big)\rightarrow \mu_{\infty}=\frac{\kappa}{\sqrt{\sigma_{\infty}^2+\kappa^2}}<1.
		\end{equation*} 
	\end{theorem}

	\subsection{Semismooth Newton algorithm for solving (\ref{equa:wk1})}
	For any given $(\tilde{w},\tilde{u})\in \mathbb{R}^t\times \mathbb{R}^t$ and $\sigma>0$, we apply the semismooth Newton algorithm to solve Problem (\ref{equa:wk1}), in which $\varphi(\tilde{w})$ can be further manipulated as
	\begin{equation*}
		\varphi(\bar{w}) 
		= \frac{1}{2} \|y-X_{[m]\mathcal{J}(\gamma)}\bar{w}\|^2-\frac{1}{2\sigma}\|\tilde{u}\|^2+\sigma\mathcal{H}_{\chi_{\bar{\mathcal{C}}}}(\bar{w}+\tilde{u}/\sigma)+\frac{1}{2\sigma}\|\bar{w}-\tilde{w}\|^2,
	\end{equation*}
	By the convexity and differentiability of Moreau envelope $\mathcal{H}_{\chi_{\bar{\mathcal{C}}}}(\cdot)$, one concludes that $\varphi(\cdot)$ is strongly convex and continuously differentiable. Hence, the unique optimal solution to (\ref{equa:wk1}) can be obtained by solving nonsmooth equations:
	\begin{equation}\label{equa:nonsmoequ}
		\begin{split}
			0= \nabla\varphi(\bar{w})
			=& X_{[m]\mathcal{J}(\gamma)}^\top(X_{[m]\mathcal{J}(\gamma)}\bar{w}-y)+\sigma \bar{w}+\tilde{u}-\sigma\Pi_{\bar{\mathcal{C}}} (\bar{w}+\tilde{u}/\sigma)+(\bar{w}-\tilde{w})/\sigma\\
			=& X_{[m]\mathcal{J}(\gamma)}^\top(X_{[m]\mathcal{J}(\gamma)}\bar{w}-y)+(\sigma+\frac{1}{\sigma})\bar{w}+\tilde{u}-\sigma\Pi_{\bar{\mathcal{C}}} (\bar{w}+\tilde{u}/\sigma)-\tilde{w}/\sigma.
		\end{split}
	\end{equation}
	Since $\Pi_{\bar{\mathcal{C}}}(\cdot)$ is piecewise affine, it follows from \cite{Facchinei2003Finite} that $\Pi_{\bar{\mathcal{C}}}(\cdot)$ is strongly semismooth. Thus, $\nabla\varphi(\cdot)$ is strongly semismooth, which allows us to develop a semismooth Newton algorithm for nonsmooth equations (\ref{equa:nonsmoequ}) and expect to obtain a fast superlinear or even quadratic convergence rate. When the semismooth Newton algorithm is derived, one needs to characterize the Clarke subdifferential \cite[Definition 2.6.1]{Clarke1983Optimization} of $\nabla\varphi$, that is the generalized Hessian for $\varphi$ denoted by $\partial^2\varphi$. For any $N\in {\cal N}_{\mathcal{C}}({\rm Vec}(\mathcal{Q}^\dagger(\bar{w}+\tilde{u}/\sigma)))$, let $\bar{N}:=N_{\mathcal{J}(\gamma)}\in \mathbb{R}^{t\times t}$, where ${\cal N}_{\mathcal{C}}$ is given in (\ref{equa:hs_jacobi}). By the definition of the generalized HS-Jacobian, we obtain that $\bar{N}$ is an arbitrary element of the generalized HS-Jacobian of $\Pi_{\bar{\mathcal{C}}}$ at $\bar{w}+\tilde{u}/\sigma$, denoted by $\bar{{\cal N}}_{\bar{\mathcal{C}}}(\bar{w}+\tilde{u}/\sigma)$. Since $\nabla\varphi(\cdot)$ is Lipschitz continuous, the multifunction as a surrogate generalized Hessian for $\varphi$ is well defined by
	\begin{equation}\label{Cal_M}
		\mathcal{M}(\bar{w}): = \{ M\in \mathbb{R}^{t\times t}\mid M:=X_{[m]\mathcal{J}(\gamma)}^\top X_{[m]\mathcal{J}(\gamma)}+\sigma(I_{t}-\bar{N})+\frac{1}{\sigma} I_{t}, \bar{N}\in \bar{{\cal N}}_{\bar{\mathcal{C}}}(\bar{w}+\tilde{u}/\sigma) \}, \forall \bar{w}\in \mathbb{R}^{t}.
	\end{equation}
	Note that for any $M\in \mathcal{M}(\bar{w})$, there exists $\bar{N}\in \bar{{\cal N}}_{\bar{\mathcal{C}}}(\bar{w}+\tilde{u}/\sigma)$ such that $M=X_{[m]\mathcal{J}(\gamma)}^\top X_{[m]\mathcal{J}(\gamma)}+\sigma(I_{t}-\bar{N})+\frac{1}{\sigma} I_{t}$. Since $I_{t}-\bar{N}$ is symmetric and positive semidefinite, it follows that $M$ is a symmetric and positive definite matrix. 
	
	With the above discussion, we intend to apply the semismooth Newton ({\sc Ssn}) algorithm to solve the inner problem (\ref{equa:wk1}) in Algorithm \ref{algo:palm}. The corresponding algorithmic framework and convergence analysis are presented as follows.
	
	\begin{algorithm}	
		\caption{\small {\bf A semismooth Newton algorithm for solving (\ref{equa:nonsmoequ})}} \label{Ssn}
		\hspace*{0.01in} \raggedright {\bf Input:} $\bar{w}^0 \in \mathbb{R}^t$, $\varrho \in (0, 1/2)$, $\nu \in (0, 1)$, $\tau \in (0,1]$, $\varpi \in (0, 1)$. Iterate the following steps for $j=0,1,\ldots$.\\
		
		\begin{algorithmic}[1]
			\STATE  Choose an element $M_j \in \mathcal{M}(\bar{w}^j)$. Apply the conjugate gradient (CG)
			method to find an approximate solution $h^j\in\mathbb{R}^t$ to 
			\begin{equation}\label{linsys}
				M_j h=-\nabla \varphi(\bar{w}^j)
			\end{equation}
			such that 
			\begin{equation}\label{cgstop}
				\| M_j h^j  + \nabla \varphi(\bar{w}^j)\|\leq\min\{\nu, \| \nabla \varphi(\bar{w}^j)\|^{1+\tau}\}.
			\end{equation}
			
			\STATE  Set $\alpha_j = \varpi^{\iota_j}$, where $\iota_j$ is the smallest nonnegative integer $\iota$ such that 
			\begin{equation*}
				\varphi(\bar{w}^j + \varpi^{\iota} h^j)\leq \varphi(\bar{w}^j)+ \varrho \varpi^{\iota}\langle \nabla \varphi(\bar{w}^j), h^j \rangle.
			\end{equation*}
			
			\STATE Update $\bar{w}^{j+1} = \bar{w}^j + \alpha_jh^j$, $j\leftarrow j+1$, and go to step 1.
		\end{algorithmic}
	\end{algorithm}	
	
	\begin{theorem}
		Assume that $\{\bar{w}^j\}$ is the infinite sequence generated by Algorithm \ref{Ssn}. Then, $\{\bar{w}^j\}$ converges globally to the unique optimal solution $\hat{w}$ of Problem (\ref{equa:wk1}) and the rate of convergence is at least superlinear with
		\begin{equation*}
			\|\bar{w} ^{j+1}-\hat{w}\|=O(\|\bar{w} ^{j}-\hat{w}\|^{1+\tau}),
		\end{equation*}
		where $\tau$ is the parameter given in Algorithm \ref{Ssn}.
	\end{theorem}
	\begin{proof}
		From \cite[ Proposition 3.3]{Zhao2010A}, we know that for any $j\geq0$, $h^j$ is always a descent direction as long as $\varphi(\bar{w}^j)\neq0$ and Algorithm \ref{Ssn} is well defined. Since the coercive property of $\varphi$ holds (cf. \cite{Rockafellar1974Conjugate}), the infinite sequence $\{\bar{w}^j\}$ is bounded. Then, by making use of the standard convergence analysis for Newton's method in \cite[Theorem 6.3.3]{Dennis1983Numerical}, we derive that
		\begin{equation*}
			\lim\limits_{j\rightarrow\infty} \nabla\varphi(\bar{w}^j)=0,
		\end{equation*}
		which, combined with the strong convexity of $\varphi(\cdot)$ and the boundedness of $\{\bar{w}^j\}$, yields that $\bar{w}^j\rightarrow\hat{w}$ and $\hat{w}$ is the unique optimal solution to (\ref{equa:wk1}). 
		
		Recall that any element $M\in \mathcal{M}(\bar{w})$ is positive definite, hence for all $j$ sufficiently large, $\|M_j\|_F^{-1}$ is uniformly bounded and the conjugate gradient method can find $h^j=-\nabla\varphi(\bar{w}^j)/M_j$ with the condition (\ref{cgstop}). Since $\Pi_{\Omega}(\cdot)$ is strongly semismooth, we follow the proofs of \cite[Theorem 3.5]{Zhao2010A} to conclude that for any $j$ sufficiently large,
		\begin{equation}\label{uequ}
			\|\bar{w}^{j}+h^j-\hat{w}\|= O(\|\bar{w}^{j}-\hat{w}\|^{1+\tau}) 
		\end{equation}
		and $-\nabla\varphi(\bar{w}^j)^\top h^j\geq\hat{\rho}\|h^j\|^2$ for some constant $\hat{\rho}>0$. Besides, one knows that, by \cite[Theorem 3.3 and Remark 3.4]{Facchinei1995Minimization}, there exists an integer $j_0$ such that for any $j\ge j_0$,
		\begin{equation*}
			\varphi(\bar{w}^j + h^j)\leq \varphi(\bar{w}^j)+ \varrho\langle \nabla \varphi(\bar{w}^j), h^j \rangle,
		\end{equation*}
		which implies that $\bar{w}^{j+1} = \bar{w}^j+h^j$ for all $j\ge j_0$. Thus, combining this with (\ref{uequ}) leads to the desired result.
	\end{proof}

	\section{Numerical Experiments}\label{sec:experiments}
	In this section, we implement the AS-{\sc Ssnpal} algorithm to solve multi-task Lasso problems on synthetic and real data sets. To illustrate the performance of the AS-{\sc Ssnpal} algorithm, we compare it to some other algorithms, including the semismooth Newton proximal augmented Lagrangian algorithm ({\sc Ssnpal}), the alternating direction method of multipliers (ADMM), and the AS strategy with the ADMM algorithm (AS-ADMM). All numerical experiments are conducted in MATLAB R2020b on an HP desktop computer with Intel(R) Core(TM) i5-10500 CPU @ 3.10 GHz and 8 GB RAM.
	
	\subsection{Some other algorithms}
	In this subsection, we shall introduce the {\sc Ssnpal} algorithm, the ADMM algorithm, and the AS strategy with the ADMM algorithm in solving the multi-task Lasso problem.
	
	The first algorithm is the {\sc Ssnpal} algorithm, which is similar to our proposed algorithm, but omits the use of the adaptive sieving reduction strategy. The {\sc Ssnpal} algorithm considers the following equivalent form of Problem (\ref{mtlproblem1}):
	\begin{equation}\label{prob:Ssnpal}
		\underset{W,Z} {\rm min}\ \Big\{ \frac{1}{2} \|y-\mathcal{X}(W)\|^2+\chi_{\Omega_{\gamma}}(Z) \mid W-Z=0\Big\}.
	\end{equation}
	Given $\sigma>0$, the augmented Lagrangian function for Problem (\ref{prob:Ssnpal}) is given by
	\begin{equation*}
		\widetilde{\mathcal{L}}_\sigma(W,Z;U):=\frac{1}{2} \|y-\mathcal{X}(W)\|^2+\chi_{\Omega_{\gamma}}(Z)+\langle U,W-Z\rangle+\frac{\sigma}{2}\|W-Z\|_F^2.
	\end{equation*}
	When the {\sc Ssnpal} algorithm is applied to Problem (\ref{prob:Ssnpal}), the subproblem at the $k-$th iteration is to solve: 
	\begin{equation*}
		\underset{W}{\rm min} \ \Big\{ \Psi_k(W):=\inf_Z \widetilde{\mathcal{L}}_{\sigma_k}(W,Z;U^k)+\frac{1}{2\sigma_k}\|
		W-\widetilde{W}^k\|^2_F\Big\}
	\end{equation*}
	for a given $\widetilde{W}^k\in\mathbb{R}^{d\times n}$. Note that $\Psi_k(W)$ can be further manipulated as
	\begin{equation*}
		\begin{split} 
			\Psi_k(W)
			=& \frac{1}{2} \|y-\mathcal{X}(W)\|^2-\frac{1}{2\sigma_k}\|U^k\|_F^2+ \sigma_k\inf_Z \Big\{\chi_{\Omega_{\gamma}}(Z)/\sigma_k+\frac{1}{2}\|Z-(W+U^k/\sigma_k)\|_F^2\Big\} \\
            &+\frac{1}{2\sigma_k}\|W-\widetilde{W}^k\|_F^2\\
			=& \frac{1}{2} \|y-\mathcal{X}(W)\|^2-\frac{1}{2\sigma_k}\|U^k\|_F^2+\sigma_k\mathcal{H}_{\chi_{\Omega_{\gamma}}}(W+U^k/\sigma_k)+\frac{1}{2\sigma_k}\|W-\widetilde{W}^k\|_F^2.
		\end{split}
	\end{equation*}
	Thus, we obtain 
	\begin{equation*}
		W^{k+1}\in \underset{W\in\mathbb{R}^{d\times n}}{\rm arg \ min} \{\Psi_k(W)\} \ {\rm and}  \ Z^{k+1}=\Pi_{\Omega_\gamma}(W^{k+1}+U^k/\sigma_k).
	\end{equation*}
	Since $\Psi_k$ is strongly convex and continuously differentiable, $W^{k+1}$ can be achieved by solving the following nonsmooth equations:
	\begin{equation*}
		\begin{split}
			0=\nabla\Psi(W)
			=&\mathcal{X}^*(\mathcal{X}(W)-y)+(\sigma_k +\frac{1}{\sigma_k}) W+U^k-\sigma_k \Pi_{\Omega_\gamma}(W+U^k/\sigma_k)-\widetilde{W}^k/\sigma_k.
		\end{split}
	\end{equation*} 
	Then, for given $W\in\mathbb{R}^{d\times n}$, we can replace the Clarke generalized Jacobian for $\nabla\Psi_k$ with the multifunction $\widetilde{\mathcal{M}}:\mathbb{R}^{d\times n}\rightrightarrows \mathbb{R}^{nd\times nd}$ defined by
	\begin{align*}
		\widetilde{\mathcal{M}}(W)(H):= \{\mathcal{X}^*\mathcal{X}(H)+(\sigma_k+\frac{1}{\sigma_k})H-\sigma_k {\rm Mat}(\widetilde{N}{\rm Vec}(H))\mid
		\widetilde{N} \in {\cal N}_{\mathcal{C}}({\rm Vec}(W+U^k/\sigma_k)) \}
	\end{align*}
	for any $H\in\mathbb{R}^{d\times n}$, where ${\cal N}_{\mathcal{C}}$ is defined in (\ref{equa:hs_jacobi}). We then outline the detailed steps of the {\sc Ssnpal} algorithm for solving (\ref{prob:Ssnpal}) in Algorithm \ref{algo:Ssnpal}.

	\begin{algorithm}
		\caption{\small {\bf A semismooth Newton proximal augmented Lagrangian algorithm for (\ref{prob:Ssnpal})}}\label{algo:Ssnpal}
		\hspace*{0.01in} \raggedright {\bf Input:} Choose an initial point $(W^0,Z^0,U^0)\in\mathbb{R}^{d\times n} \times \mathbb{R}^{d\times n} \times \mathbb{R}^{d\times n}$. Given parameter $\sigma_0>0$, $\varrho \in (0, 1/2)$, $\nu \in (0, 1)$, $\tau \in (0,1]$, and $\varpi \in (0, 1)$. Set $k=0$. \\
		\hspace*{0.01in} \raggedright {\bf Output:} An approximate optimal solution $(\hat{W},\hat{Z},\hat{U})$. \\
		\begin{algorithmic}[1]
			
			\WHILE{Stopping criteria are not met}
			
            \STATE Specify an initial value $W_0=W^k$. Iterate the following steps with $j=0,1,\ldots$.
			\WHILE{Stopping criteria for the inner loop are not met}
			
            \STATE Choose $\widetilde{N} \in {\cal N}_{\mathcal{C}}({\rm Vec}(W_j+U^k/\sigma_k)) \}$. For $H\in\mathbb{R}^{d\times n}$, let $\widetilde{M}_jH:=\mathcal{X}^*\mathcal{X}(H)+\sigma_k(H-{\rm Mat}(\widetilde{N}{\rm Vec}(H)))+H/{\sigma_k} $. Apply the direct method or the conjugate gradient algorithm to solve equations
			\begin{equation*}
				\widetilde{M}_jH+\nabla\Psi(W_j)=0
			\end{equation*}
			such that
			\begin{equation*}
				\|\widetilde{M}_j H^j  + \nabla\Psi(W_j)\|\leq\min\{\nu, \| \nabla\Psi(W_j)\|^{1+\tau}\}.
			\end{equation*}
			
            \STATE Set $\alpha_j = \varpi^{\iota_j}$, where $\iota_j$ is the smallest nonnegative integer $\iota$ such that 
			\begin{equation*}
				\Psi(W_j + \varpi^{\iota} H^j)\leq \Psi(W_j)+ \varrho \varpi^{\iota}\langle \nabla \Psi(W_j), H^j \rangle.
			\end{equation*}
			
            \STATE Update $W_{j+1} = W_j + \alpha_jH^j$ and $W^{k+1}=W_{j+1}$.
			\STATE $j++$.
			\ENDWHILE
			
            \STATE Compute
			\begin{equation*}
				Z^{k+1}=\Pi_{\Omega_\gamma}( W^{k+1}+Z^k/\sigma_k), \ U^{k+1}=U^k+\sigma_k (W^{k+1}-Z^{k+1}).
			\end{equation*}
			
            \STATE Update $\sigma_{k+1} \uparrow \sigma_\infty\leq \infty$, $\hat{W}=W^{k+1}$, $\hat{Z}=Z^{k+1}$, and $\hat{U}=U^{k+1}$.
			\STATE $k:=k+1$.
			\ENDWHILE
		\end{algorithmic}
	\end{algorithm}
	Next, we introduce the ADMM algorithm to solve Problem (\ref{prob:Ssnpal}). The general steps of the ADMM algorithm are as follows:
	\begin{align*}
		W^{k+1}&\approx \underset{W}{\rm arg \ min} \ \{\widetilde{\mathcal{L}}_{\sigma}(W,Z^k;U^k)\},\\
		Z^{k+1}&= \underset{Z}{\rm arg \ min} \ \{\widetilde{\mathcal{L}}_{\sigma}(W^{k+1},Z;U^k)\},\\
		U^{k+1} &= U^k+\kappa\sigma ( W^{k+1}-Z^{k+1}),
	\end{align*}
	where $\kappa>0$ is a given step length. By the continuous differentiability of $\widetilde{\mathcal{L}}_{\sigma}$, one solves the following linear system to obtain $W^{k+1}$:
	\begin{equation*}
		(\mathcal{X}^*\mathcal{X}+\sigma I_d)W=\mathcal{X}^*(y)+\sigma(Z^k-U^k/\sigma).
	\end{equation*}
	Here, the CG algorithm is considered since $\mathcal{X}^*\mathcal{X}$ is a positive semidefinite operator. The framework of the ADMM algorithm is summarized in Algorithm \ref{algo:ADMM}.
	\begin{algorithm}
		\caption{\small {\bf An alternating direction method of multipliers for (\ref{prob:Ssnpal}) }} \label{algo:ADMM}
		\hspace*{0.01in}\raggedright {\bf Input:} $\kappa \in (0, (1+\sqrt{5})/2)$, $\sigma >0$, $(W^0,Z^0,U^0)\in\mathbb{R}^{d\times n} \times \mathbb{R}^{d\times n} \times \mathbb{R}^{d\times n} $. Set $k=0$.  \\
		\begin{algorithmic}[1]
			\STATE  Apply the CG algorithm to obtain $W^{k+1}$ such that
			\begin{equation*}
				W^{k+1}\approx \underset{W}{\rm arg \min} \ \{\widetilde{\mathcal{L}}_{\sigma}(W,Z^k;U^k)\}.
			\end{equation*}
			
			\STATE  Compute
			$
			Z^{k+1}=\Pi_{\Omega_\gamma}(W^{k+1}+Z^k/\sigma).
			$

			\STATE Compute $U^{k+1} = U^k+\kappa\sigma ( W^{k+1}-Z^{k+1}).$	
			\STATE Set $k \leftarrow k+1$, and go to step 1.
		\end{algorithmic}
	\end{algorithm}

	We also test the AS strategy with the ADMM algorithm and denote the resulting method by AS-ADMM. Given $\sigma>0$, we define
	\begin{equation*}
		\psi(\bar{w},\bar{z},\bar{u})=\frac{1}{2} \|y-X_{[m]\mathcal{J}(\gamma)}\bar{w}\|^2-\frac{1}{2\sigma}\|\bar{u}\|^2+\chi_{\bar{\mathcal{C}}}(\bar{z})+\frac{\sigma}{2}\|\bar{z}-(\bar{w}+\bar{u}/\sigma)\|^2.
	\end{equation*}
	Similar to discussions in Section \ref{sec:Ssnpal}, an optimal solution to Problem (\ref{prob:problemp'}) can be obtained by solving 
	\begin{equation*}
		\underset{\bar{w},\bar{z},\bar{u}}{\min} \ \{\psi(\bar{w},\bar{z},\bar{u})\}.
	\end{equation*}
	
	The detailed steps of the ADMM algorithm on Problem (\ref{prob:problemp'}) are described in Algorithm \ref{algo:AS-ADMM}.
	
	\begin{algorithm}[H]
		\caption{\small {\bf An alternating direction method of multipliers for (\ref{prob:problemp'}) }} \label{algo:AS-ADMM}
		\hspace*{0.01in} \raggedright {\bf Input:} $\kappa \in (0, (1+\sqrt{5})/2)$, $\sigma >0$,  $(W^0,Z^0,U^0)\in \mathbb{R}^{d\times n}\times \mathbb{R}^{d\times n}$. Set $k=0$.  \\
		\hspace*{0.01in} \raggedright {\bf Output:} An approximate optimal solution $(\hat{W},\hat{Z},\hat{U})$.\\
		\begin{algorithmic}[1]		
			\WHILE{Stopping criteria are not met}
			\STATE  Let $\bar{w}^0=\mathcal{Q}(W^0)$ and $\bar{u}^0=\mathcal{Q}(U^0)$. Iterate the following steps with $k=0,1,\ldots$
			
			\STATE  Apply the CG algorithm to obtain $\bar{w}^{k+1}$ such that
			\begin{equation*}
				\bar{w}^{k+1}\approx \underset{\bar{w}}{\rm arg \min} \ \{\psi(\bar{w},\bar{z}^k,\bar{u}^k)\}.
			\end{equation*}
			Extend $\bar{w}^{k+1}$ to $W^{k+1}$ by $W^{k+1}=\mathcal{Q}^\dagger(\bar{w}^{k+1})$.
			
			\STATE  
			Compute 
			\begin{equation*}
				Z^{k+1}=\Pi_{\Omega_\gamma}(W^{k+1}+U^k/\sigma_k), \ 
				\bar{z}^{k+1}=\mathcal{Q}(Z^{k+1}).
			\end{equation*}

			\STATE 
			Compute 
			\begin{equation*}
				U^{k+1} = U^k+\kappa\sigma(W^{k+1}-Z^{k+1}), \ 
				\bar{u}^{k+1}=\mathcal{Q}(U^{k+1}).
			\end{equation*}
			
			\STATE 
			Update $\hat{W}=W^{k+1}$, $\hat{Z}=Z^{k+1}$ and $\hat{U}=U^{k+1}$.
			
			\STATE
			$k:=k+1$.
			\ENDWHILE
		\end{algorithmic}
	\end{algorithm} 
	In Algorithm \ref{algo:AS-ADMM}, the optimal solution $\bar{w}^{k+1}$ to the subproblem $\underset{\bar{w}}{\rm arg \min} \ \{\psi(\bar{w},\bar{z}^k,\bar{u}^k)\}$ can be obtained by solving 
	\begin{equation*}
		(X_{[m]\mathcal{J}(\gamma)}^\top X_{[m]\mathcal{J}(\gamma)}+\sigma I_t)\bar{w}=X_{[m]\mathcal{J}(\gamma)}^\top y+\sigma(\bar{z}-\bar{u}/\sigma).
	\end{equation*}

	\subsection{Stopping criteria}
	
	According to the KKT conditions (\ref{equa:kkt_P'}) for Problem (\ref{prob:problemp'}), we define the following relative residuals:
	\begin{equation*}
		{\rm Res}_1:=\frac{\|W-Z\|_F}{1+\|W\|_F+\|Z\|_F}, \ \ 
		{\rm Res}_2:=\frac{\|Z-\Pi_{\Omega_{\gamma}}(Z+U)\|_F}{1+\|Z\|_F+\|U\|_F},
	\end{equation*}
	\begin{equation*}
		{\rm Res}_3:=\frac{\|X_{[m]\mathcal{J}(\gamma)}^\top(X_{[m]\mathcal{J}(\gamma)}\bar{w}-y)+\bar{u})\|}{1+\|\bar{u}\|+\|X_{[m]\mathcal{J}(\gamma)}^\top(X_{[m]\mathcal{J}(\gamma)}\bar{w}-y)\|}.
	\end{equation*}
	In our experiments, we will start the AS-{\sc Ssnpal} and AS-ADMM algorithm with the initial point $(W,Z,U)=(0,0,0)$ and terminate them if $R_{{\rm kkt}}:={\rm max}({\rm Res}_1,{\rm Res}_2,{\rm Res}_3)\leq {\rm tol}$ holds, where ``tol" refers to a given accuracy tolerance.
	
	Likewise, we start the {\sc Ssnpal} and ADMM algorithm from the point $(W,Z,U)=(0,0,0)$ and terminate them if $R_{{\rm kkt}}:={\rm max}({\rm Res}_1,{\rm Res}_2,\widetilde{{\rm Res}}_3)\leq {\rm tol}$, in which $\widetilde{{\rm Res}}_3$ is as follows
	\begin{equation*}
		\widetilde{{\rm Res}}_3:=\frac{\|\mathcal{X}^*(\mathcal{X}(W)-y)+U\|_F}{1+\|U\|_F+\|\mathcal{X}^*(\mathcal{X}(W)-y)\|_F}.
	\end{equation*}
	
	Additionally, all tested algorithms are terminated when they achieve the preset maximum number of iterations ($200$ for AS-{\sc Ssnpal} and {\sc Ssnpal}, $30000$ for AS-ADMM and ADMM) or when the running time exceeds $2$ hours.
	
	\subsection{Parameter settings}
	In algorithm \ref{algor:AS}, we adhere to the guidelines outlined in \cite{Lin2020Adaptive} for establishing the initial active index set $I^*(\gamma_0)$. We first denote 
	\begin{align*}
		\eta_i=
		\left\{ 
		\begin{array}{ll}
			i/d, & {\rm rem}(i,d)=0,  \\ [5pt]
			i/d+1, & {\rm otherwise}, 
		\end{array} \right.
		\ i = 1,\ldots,nd
	\end{align*} 
	and 
	\begin{equation*}
		\mathcal{D}=\{i\in\{1,2,\ldots,nd\}\mid s_i {\rm \ is \ among \ the \ first} \ \lceil\sqrt{n}\rceil \ {\rm largest \ values \ in} \  \{s_1,s_2,\ldots,s_nd\} \},
	\end{equation*}
	where $\lceil\sqrt{n}\rceil$ signifies the value of rounding up $\sqrt{n}$, and $s_i=\frac{\lvert \langle X_{[m]i},y\rangle\rvert}{\|X_{[m]i}\|\|y\|}$ for $i=1,2,\ldots,nd$. Then the index set $I^*(\gamma_0)$ is obtained by 
	\begin{equation*}
		I^*(\gamma_0)=\{(i,j)\mid i=\alpha_k,\ j=\eta_k, k\in \mathcal{D} \},
	\end{equation*}
	where $\alpha_k$ is given in (\ref{alpha}).
	
	The parameters of Algorithm \ref{algo:palm}, Algorithm \ref{Ssn}, and {\sc Ssnpal} in Algorithm \ref{algo:Ssnpal}, are the same as \cite{Lin2024An}, except for the penalty parameter $\sigma$. Specifically, we choose $\sigma_0=500$ for the {\sc Ssnpal} algorithm. For the AS-{\sc Ssnpal} algorithm, we set $\sigma_0=500$, and update $\sigma_{k+1}$ as follows:
	\begin{equation}\label{sigma}
		\sigma_{k+1}=\left\{ 
		\begin{array}{lll}
			{\rm min}\{1.5\sigma_k,10^7\}, &  {\rm Res}_3^k<{\rm Res}_1^k, \\ [5pt]
			{\rm max}\{10^{-5},\varsigma\sigma_k\}, & {\rm Res}_3^k\ge{\rm Res}_1^k \ {\rm and} \  {\rm Res}_3^k>0.9{\rm Res}_3^{k-1},\ k\ge 1,   \\ [5pt]
			{\rm min}\{1.05\sigma_k,10^6\}, & {\rm otherwise},
		\end{array} \right.
	\end{equation}
	where $\sigma_k$, ${\rm Res}_1^k$, ${\rm Res}_3^k$ and $R_{{\rm kkt}}^k$ are the value of $\sigma$, ${\rm Res}_1$, ${\rm Res}_3$ and $R_{{\rm kkt}}$ respectively at the $k$-th iteration, and 
	\begin{equation*}
		\varsigma=\left\{ 
		\begin{array}{lll}
			0.5, &  {\rm Res}_1^k<0.9{\rm Res}_1^{k-1} \ {\rm and} \ R_{\rm kkt}^k<5{\rm tol}, \\ [5pt]
			0.8, & 0.9{\rm Res}_1^{k-1}\le{\rm Res}_1^k<1.1{\rm Res}_1^{k-1} \ {\rm and} \ R_{\rm kkt}^k<5{\rm tol},  \\ [5pt] 
			0.9, & {\rm otherwise}. 
		\end{array} \right.
	\end{equation*}
	
	For Algorithm \ref{algo:ADMM} and Algorithm \ref{algo:AS-ADMM}, we choose the step length $\kappa=1.618$ and the penalty parameter $\sigma=100$. 
	
	\subsection{Numerical results for synthetic data}
	In this subsection, we conduct a performance comparison among all algorithms (AS-{\sc Ssnpal}, AS-ADMM, {\sc Ssnpal}, and ADMM) applied to solve multi-task Lasso problems on synthetic data sets. The experiment setup is as follows: set $(m,nd) = (2560i,720i), \ i=1,2,\ldots,10$. For each $(m,nd)$, there are $n = 20i$ tasks, $d = 36$ features and $m_j = 128, \ j=1,2,\ldots,n$ samples. The elements of the data matrix $X^j\in\mathbb{R}^{m_j\times d}$ for the $j$-th task are drawn randomly from a Gaussian distribution $\mathcal{N}(0,37)$. For the sparse matrix $W\in \mathbb{R}^{d\times n}$, its entries are generated randomly from the normal distribution $\mathcal{N}(0,1)$, with 60\% of its components randomly set to zero. The regression coefficient vector $\hat{y}^j\in\mathbb{R}^{m_j}$ is then obtained from $\hat{y}^j=X^j\hat{w}^j+\delta^j$, where $\delta^j \sim\mathcal{N}(0,I_{m_j})$. We choose $\gamma = 0.01, 0.03, 0.05$ and randomly generate $10$ different instances for each $(m,nd)$. The average numerical results of AS-{\sc Ssnpal}, AS-ADMM, {\sc Ssnpal}, and ADMM under ${\rm tol}=10^{-3}$ and ${\rm tol}=10^{-6}$ are listed in Table \ref{TableRand:p_e3} and \ref{TableRand:p_e6}, respectively. To be specific, we report the cumulative computation time in seconds (Time) and the relative KKT residuals ($R_{\rm kkt}$).

	\begin{center}
		\setlength{\tabcolsep}{8.8pt}{
			\begin{longtable}{|c|c|cccc|cccc|}
				\captionsetup{width=0.85\textwidth}
				\caption{The performance of AS-{\sc Ssnpal}, AS-ADMM, {\sc Ssnpal} and ADMM on synthetic data. In the table, we terminate algorithms when $R_{{\rm kkt}}\le 10^{-3}$. ``a1"=AS-{\sc Ssnpal}, ``a2"=AS-ADMM, ``a3"={\sc Ssnpal}, ``a4"=ADMM. Times are shown in seconds. } \label{TableRand:p_e3}\\
				
				\hline 
				\multirow{2}*{i} &\multirow{2}*{$\gamma$}&
				\multicolumn{4}{c|}{Time} &
				\multicolumn{4}{c|}{$R_{kkt}$} \\ \cline{3-10}
				& &  \multicolumn{1}{c}{a1}& \multicolumn{1}{c}{a2}& \multicolumn{1}{c}{a3}& \multicolumn{1}{c|}{a4} &\multicolumn{1}{c}{a1}& \multicolumn{1}{c}{a2}& \multicolumn{1}{c}{a3}& \multicolumn{1}{c|}{a4}  \\ \hline
				\endfirsthead
				
				\multicolumn{10}{c}{{\bfseries \tablename\ \thetable{} -- continued from previous page}} \\
				
				\hline
				\multirow{2}*{i} &\multirow{2}*{$\gamma$}&
				\multicolumn{4}{c|}{Time} &
				\multicolumn{4}{c|}{$R_{kkt}$} \\ \cline{3-10}
				& &  \multicolumn{1}{c}{a1}& \multicolumn{1}{c}{a2}& \multicolumn{1}{c}{a3}& \multicolumn{1}{c|}{a4} &\multicolumn{1}{c}{a1}& \multicolumn{1}{c}{a2}& \multicolumn{1}{c}{a3}& \multicolumn{1}{c|}{a4}  \\ \hline
				\endhead
				
				\multicolumn{10}{|r|}{\textbf{Continued on next page}} \\
				\hline
				\endfoot
				
				\hline
				\endlastfoot
				\multirow{3}*{1}	 & 0.01 	 & 0.06 	 & 2.79 	 & 0.08 	 & 0.06 	 & 1.8e-04	 & 1.0e-03	 & 8.0e-04	 & 9.3e-04\\  
	 & 0.03 	 & 0.13 	 & 5.45 	 & 0.12 	 & 0.12 	 & 1.7e-04	 & 1.0e-03	 & 8.0e-04	 & 9.3e-04\\  
	 & 0.05 	 & 0.17 	 & 8.06 	 & 0.17 	 & 0.20 	 & 1.6e-04	 & 1.0e-03	 & 8.0e-04	 & 9.3e-04\\  
\hline\multirow{3}*{2}	 & 0.01 	 & 0.05 	 & 4.58 	 & 0.16 	 & 0.75 	 & 5.0e-04	 & 1.0e-03	 & 8.0e-04	 & 9.3e-04\\  
	 & 0.03 	 & 0.10 	 & 8.57 	 & 0.31 	 & 1.46 	 & 4.7e-04	 & 1.0e-03	 & 8.0e-04	 & 9.3e-04\\  
	 & 0.05 	 & 0.18 	 & 12.51 	 & 0.47 	 & 2.15 	 & 4.3e-04	 & 1.0e-03	 & 8.0e-04	 & 9.3e-04\\  
\hline\multirow{3}*{3}	 & 0.01 	 & 0.07 	 & 4.46 	 & 0.53 	 & 2.17 	 & 2.7e-04	 & 1.0e-03	 & 8.0e-04	 & 9.3e-04\\  
	 & 0.03 	 & 0.15 	 & 8.88 	 & 0.98 	 & 4.33 	 & 2.5e-04	 & 1.0e-03	 & 8.0e-04	 & 9.3e-04\\  
	 & 0.05 	 & 0.21 	 & 13.26 	 & 1.46 	 & 6.49 	 & 2.4e-04	 & 1.0e-03	 & 8.0e-04	 & 9.3e-04\\  
\hline\multirow{3}*{4}	 & 0.01 	 & 0.10 	 & 5.90 	 & 0.98 	 & 4.51 	 & 3.6e-04	 & 1.0e-03	 & 8.0e-04	 & 9.3e-04\\  
	 & 0.03 	 & 0.21 	 & 11.72 	 & 1.96 	 & 9.00 	 & 3.4e-04	 & 1.0e-03	 & 8.0e-04	 & 9.3e-04\\  
	 & 0.05 	 & 0.32 	 & 17.43 	 & 3.02 	 & 13.50 	 & 3.1e-04	 & 1.0e-03	 & 8.0e-04	 & 9.3e-04\\  
\hline\multirow{3}*{5}	 & 0.01 	 & 0.11 	 & 6.82 	 & 1.79 	 & 8.46 	 & 5.2e-04	 & 1.0e-03	 & 8.0e-04	 & 9.3e-04\\  
	 & 0.03 	 & 0.22 	 & 13.54 	 & 3.56 	 & 16.92 	 & 4.8e-04	 & 1.0e-03	 & 8.0e-04	 & 9.3e-04\\  
	 & 0.05 	 & 0.30 	 & 20.14 	 & 5.34 	 & 25.37 	 & 4.4e-04	 & 1.0e-03	 & 8.0e-04	 & 9.3e-04\\  
\hline\multirow{3}*{6}	 & 0.01 	 & 0.14 	 & 7.82 	 & 2.98 	 & 12.91 	 & 5.8e-04	 & 1.0e-03	 & 8.0e-04	 & 9.3e-04\\  
	 & 0.03 	 & 0.23 	 & 15.50 	 & 5.92 	 & 25.68 	 & 5.3e-04	 & 1.0e-03	 & 8.0e-04	 & 9.3e-04\\  
	 & 0.05 	 & 0.34 	 & 23.08 	 & 8.89 	 & 38.51 	 & 4.9e-04	 & 1.0e-03	 & 8.0e-04	 & 9.3e-04\\  
\hline\multirow{3}*{7}	 & 0.01 	 & 0.13 	 & 7.39 	 & 4.56 	 & 18.36 	 & 5.6e-04	 & 1.0e-03	 & 8.0e-04	 & 9.3e-04\\  
	 & 0.03 	 & 0.23 	 & 14.62 	 & 9.08 	 & 36.44 	 & 5.1e-04	 & 1.0e-03	 & 8.0e-04	 & 9.3e-04\\  
	 & 0.05 	 & 0.34 	 & 21.80 	 & 13.56 	 & 54.10 	 & 4.6e-04	 & 1.0e-03	 & 8.0e-04	 & 9.3e-04\\  
\hline\multirow{3}*{8}	 & 0.01 	 & 0.14 	 & 7.52 	 & 6.46 	 & 24.09 	 & 4.5e-04	 & 1.0e-03	 & 8.0e-04	 & 9.3e-04\\  
	 & 0.03 	 & 0.26 	 & 14.91 	 & 12.84 	 & 48.15 	 & 4.1e-04	 & 1.0e-03	 & 8.0e-04	 & 9.3e-04\\  
	 & 0.05 	 & 0.40 	 & 22.22 	 & 19.24 	 & 72.24 	 & 3.8e-04	 & 1.0e-03	 & 8.0e-04	 & 9.3e-04\\  
\hline\multirow{3}*{9}	 & 0.01 	 & 0.14 	 & 8.06 	 & 8.90 	 & 32.52 	 & 4.5e-04	 & 1.0e-03	 & 8.0e-04	 & 9.3e-04\\  
	 & 0.03 	 & 0.27 	 & 16.06 	 & 17.70 	 & 66.04 	 & 4.1e-04	 & 1.0e-03	 & 8.0e-04	 & 9.3e-04\\  
	 & 0.05 	 & 0.42 	 & 23.94 	 & 26.47 	 & 98.43 	 & 3.8e-04	 & 1.0e-03	 & 8.0e-04	 & 9.3e-04\\  
\hline\multirow{3}*{10}	 & 0.01 	 & 0.15 	 & 8.55 	 & 11.98 	 & 41.74 	 & 4.7e-04	 & 1.0e-03	 & 8.0e-04	 & 9.3e-04\\  
	 & 0.03 	 & 0.27 	 & 17.04 	 & 24.16 	 & 83.60 	 & 4.3e-04	 & 1.0e-03	 & 8.0e-04	 & 9.3e-04\\  
	 & 0.05 	 & 0.44 	 & 25.45 	 & 36.33 	 & 125.41 	 & 4.0e-04	 & 1.0e-03	 & 8.0e-04	 & 9.3e-04
		\end{longtable}}
	\end{center}
	
	As shown in Table  \ref{TableRand:p_e3}, all tested algorithms successfully solve all instances for $\gamma = 0.01, 0.03, 0.05$ with the desired accuracy. A careful comparison reveals that AS-{\sc Ssnpal} performs competitively with {\sc Ssnpal} in some specific cases, such as the small-scale instance $i=1$. However, for moderate to large-scale instances, AS-{\sc Ssnpal} is approximately $1.5$ times faster than {\sc Ssnpal}. Additionally, the cumulative time required by AS-{\sc Ssnpal} to generate solution paths for ${0.01, 0.03, 0.05}$ is significantly less than that of AS-ADMM and ADMM. For instance, in case $7$, AS-{\sc Ssnpal} generates the solution paths in just $0.5$ seconds, while AS-ADMM and ADMM require about $27$ and $64$ seconds, respectively. In fact, AS-{\sc Ssnpal} consistently solves each solution path in under $1$ second.
	
	\begin{center}
		\setlength{\tabcolsep}{8.0pt}{
			\begin{longtable}{|c|c|cccc|cccc|}
				\captionsetup{width=0.9\textwidth}
				\caption{The performance of AS-{\sc Ssnpal}, AS-ADMM, {\sc Ssnpal} and ADMM on synthetic data. In the table, we terminate algorithms when $R_{{\rm kkt}}\le 10^{-6}$.  ``a1"=AS-{\sc Ssnpal}, ``a2"=AS-ADMM, ``a3"={\sc Ssnpal}, ``a4"=ADMM. Times are shown in seconds.} \label{TableRand:p_e6}\\
				
				\hline 
				\multirow{2}*{i} &\multirow{2}*{$\gamma$}&
				\multicolumn{4}{c|}{Time} &
				\multicolumn{4}{c|}{$R_{kkt}$} \\ \cline{3-10}
				& &  \multicolumn{1}{c}{a1}& \multicolumn{1}{c}{a2}& \multicolumn{1}{c}{a3}& \multicolumn{1}{c|}{a4} &\multicolumn{1}{c}{a1}& \multicolumn{1}{c}{a2}& \multicolumn{1}{c}{a3}& \multicolumn{1}{c|}{a4}  \\ \hline
				\endfirsthead
				
				\multicolumn{10}{c}{{\bfseries \tablename\ \thetable{} -- continued from previous page}} \\
				
				\hline
				\multirow{2}*{i} &\multirow{2}*{$\gamma$}&
				\multicolumn{4}{c|}{Time} &
				\multicolumn{4}{c|}{$R_{kkt}$} \\ \cline{3-10}
				& &  \multicolumn{1}{c}{a1}& \multicolumn{1}{c}{a2}& \multicolumn{1}{c}{a3}& \multicolumn{1}{c|}{a4} &\multicolumn{1}{c}{a1}& \multicolumn{1}{c}{a2}& \multicolumn{1}{c}{a3}& \multicolumn{1}{c|}{a4}  \\ \hline
				\endhead
				
				\multicolumn{10}{|r|}{\textbf{Continued on next page}} \\
				\hline
				\endfoot
				
				\hline
				\endlastfoot
				
				\multirow{3}*{1}	 & 0.01 	 & 0.17 	 & 6.81 	 & 0.28 	 & 6.46 	 & 2.1e-07	 & 1.0e-06	 & 8.6e-07	 & 1.0e-06\\  
	 & 0.03 	 & 0.22 	 & 13.60 	 & 0.53 	 & 12.85 	 & 1.9e-07	 & 1.0e-06	 & 8.6e-07	 & 1.0e-06\\  
	 & 0.05 	 & 0.27 	 & 20.32 	 & 0.77 	 & 19.21 	 & 1.7e-07	 & 1.0e-06	 & 8.6e-07	 & 1.0e-06\\  
\hline\multirow{3}*{2}	 & 0.01 	 & 0.11 	 & 11.23 	 & 0.34 	 & 115.69 	 & 8.2e-08	 & 1.0e-06	 & 8.5e-07	 & 1.0e-06\\  
	 & 0.03 	 & 0.19 	 & 22.32 	 & 0.68 	 & 226.52 	 & 6.7e-08	 & 1.0e-06	 & 8.5e-07	 & 1.0e-06\\  
	 & 0.05 	 & 0.36 	 & 33.54 	 & 1.01 	 & 340.26 	 & 6.9e-08	 & 1.0e-06	 & 8.5e-07	 & 1.0e-06\\  
\hline\multirow{3}*{3}	 & 0.01 	 & 0.28 	 & 21.85 	 & 0.82 	 & 509.27 	 & 1.6e-07	 & 1.0e-06	 & 8.6e-07	 & 1.0e-06\\  
	 & 0.03 	 & 0.69 	 & 43.48 	 & 1.64 	 & 1135.47 	 & 1.5e-07	 & 1.0e-06	 & 8.6e-07	 & 1.0e-06\\  
	 & 0.05 	 & 0.84 	 & 64.78 	 & 2.50 	 & 1645.94 	 & 1.3e-07	 & 1.0e-06	 & 8.6e-07	 & 1.0e-06\\  
\hline\multirow{3}*{4}	 & 0.01 	 & 0.23 	 & 23.32 	 & 1.61 	 & 1022.91 	 & 7.9e-08	 & 1.0e-06	 & 8.6e-07	 & 1.0e-06\\  
	 & 0.03 	 & 0.36 	 & 46.40 	 & 3.12 	 & 2178.25 	 & 6.0e-08	 & 1.0e-06	 & 8.6e-07	 & 1.0e-06\\  
	 & 0.05 	 & 0.88 	 & 69.22 	 & 4.74 	 & 3239.45 	 & 6.0e-08	 & 1.0e-06	 & 8.6e-07	 & 1.0e-06\\  
\hline\multirow{3}*{5}	 & 0.01 	 & 0.42 	 & 31.07 	 & 2.51 	 & 2058.97 	 & 1.5e-07	 & 1.0e-06	 & 8.5e-07	 & 1.0e-06\\  
	 & 0.03 	 & 0.89 	 & 62.21 	 & 5.03 	 & 3988.10 	 & 1.5e-07	 & 1.0e-06	 & 8.5e-07	 & 1.0e-06\\  
	 & 0.05 	 & 1.14 	 & 92.56 	 & 7.72 	 & 6473.85 	 & 1.3e-07	 & 1.0e-06	 & 8.5e-07	 & 1.0e-06\\  
\hline\multirow{3}*{6}	 & 0.01 	 & 0.31 	 & 46.62 	 & 3.96 	 & 2758.39 	 & 7.8e-08	 & 1.0e-06	 & 8.6e-07	 & 1.0e-06\\  
	 & 0.03 	 & 1.83 	 & 93.46 	 & 8.00 	 & 5679.36 	 & 1.5e-07	 & 1.0e-06	 & 8.6e-07	 & 1.0e-06\\  
	 & 0.05 	 & 2.24 	 & 140.56 	 & 12.07 	 & 8545.77 	 & 1.3e-07	 & 1.0e-06	 & 8.6e-07	 & 1.0e-06\\  
\hline\multirow{3}*{7}	 & 0.01 	 & 0.37 	 & 63.49 	 & 5.98 	 & 4021.68 	 & 1.4e-07	 & 1.0e-06	 & 8.5e-07	 & 1.0e-06\\  
	 & 0.03 	 & 1.26 	 & 125.94 	 & 11.91 	 & 7851.84 	 & 1.6e-07	 & 1.0e-06	 & 8.5e-07	 & 1.0e-06\\  
	 & 0.05 	 & 2.60 	 & 188.55 	 & 17.84 	 & 11905.09 	 & 1.7e-07	 & 1.0e-06	 & 8.5e-07	 & 1.0e-06\\  
\hline\multirow{3}*{8}	 & 0.01 	 & 0.76 	 & 44.53 	 & 8.51 	 & 5077.43 	 & 1.1e-07	 & 1.0e-06	 & 8.6e-07	 & 1.0e-06\\  
	 & 0.03 	 & 1.14 	 & 88.51 	 & 16.96 	 & 10419.78 	 & 8.4e-08	 & 1.0e-06	 & 8.6e-07	 & 1.0e-06\\  
	 & 0.05 	 & 1.94 	 & 131.48 	 & 25.40 	 & 17493.10 	 & 8.8e-08	 & 1.0e-06	 & 8.6e-07	 & 1.0e-06\\  
\hline\multirow{3}*{9}	 & 0.01 	 & 0.54 	 & 70.21 	 & 11.53 	 & 6635.07 	 & 9.9e-08	 & 1.0e-06	 & 8.5e-07	 & 1.0e-06\\  
	 & 0.03 	 & 2.47 	 & 140.39 	 & 22.88 	 & 13835.99 	 & 1.4e-07	 & 1.0e-06	 & 8.5e-07	 & \textbf{1.2e-06}\\  
	 & 0.05 	 & 3.09 	 & 209.79 	 & 34.19 	 & 19980.86 	 & 1.2e-07	 & 1.0e-06	 & 8.5e-07	 & 1.0e-06\\  
\hline\multirow{3}*{10}	 & 0.01 	 & 1.18 	 & 425.67 	 & 16.34 	 & 7202.35 	 & 1.2e-07	 & 1.0e-06	 & 8.6e-07	 & \textbf{1.6e-06}\\  
	 & 0.03 	 & 3.00 	 & 814.07 	 & 31.48 	 & 14403.16 	 & 1.6e-07	 & 1.0e-06	 & 8.5e-07	 & \textbf{1.6e-06}\\  
	 & 0.05 	 & 5.83 	 & 1468.62 	 & 46.76 	 & 21604.91 	 & 2.0e-07	 & 1.0e-06	 & 8.5e-07	 & \textbf{1.6e-06}
				
		\end{longtable}}
	\end{center}
	
	Table \ref{TableRand:p_e6} presents the numerical results for the AS-{\sc Ssnpal}, AS-ADMM, {\sc Ssnpal}, and ADMM algorithms with a tolerance of ${\rm tol}=10^{-6}$. The values highlighted in bold indicate that the corresponding algorithm failed to reach the desired accuracy within the preset time or maximum number of iterations. From the results in Table \ref{TableRand:p_e6}, we observe that all tested algorithms, except for ADMM, successfully solve all instances, highlighting the lower robustness of ADMM compared to AS-{\sc Ssnpal}, AS-ADMM, and {\sc Ssnpal}. When comparing the total computation time across all algorithms, our proposed AS-{\sc Ssnpal} algorithm proves to be significantly more efficient. For instance, in instance $6$, AS-{\sc Ssnpal} is approximately $63$ times faster than AS-ADMM, 5 times faster than {\sc Ssnpal}, and $3815$ times faster than ADMM. Moreover, AS-{\sc Ssnpal} consistently achieves more accurate approximate solutions than AS-ADMM, {\sc Ssnpal}, and ADMM for each tested instance.

	\subsection{Numerical results for real data} 
	In this subsection, we use real datasets to demonstrate the superior performance of our proposed algorithm in generating solution paths. We also compare its performance with that of other algorithms, including {\sc Ssnpal}, AS-ADMM, and ADMM. The data matrix $\widetilde{X}=[X^1;X^2;\ldots;X^n]\in\mathbb{R}^{m\times d}$ and the response vector $y\in\mathbb{R}^m$ are derived from two different real datasets: the LIBSVM data repository\footnote{\url{https://www.csie.ntu.edu.tw/~cjlin/libsvmtools/datasets/}} and Andreas Argyriou's homepage\footnote{\url{https://home.ttic.edu/~argyriou/code/index.html}}.

	\subsubsection{School data}
	The School dataset, originally provided by the Inner London Education Authority\footnote{\url{http://www.mlwin.com/intro/datasets.html}}, is widely used in multi-task regression studies (cf. \cite{Argyriou2008Convex,Chen2011Integrating,Chu2020Semismooth}). It contains examination scores for $15,362$ students from $139$ secondary schools. Each student's performance is characterized by $28$ attributes, such as gender and ethnic group. Consequently, the dataset consists of $n=139$ tasks, each with a varying number of observations, totaling $m=15,362$. Each observation includes $d=28$ features. In our experiments, we set the number of observations for each task to $m_j=110$ for $j=1,2,\ldots,138$, and $m_{139}=182$. Additionally, we standardized the data so that each column of $(\widetilde{X},y)$ has a mean of $0$ and a variance of $1$.
	
	\begin{center}
		\setlength{\tabcolsep}{6.5pt}{
			\begin{longtable}{|c|cccc|cccc|}
				\captionsetup{width=0.9\textwidth}
				\caption{The performance of AS-{\sc Ssnpal}, AS-ADMM, {\sc Ssnpal} and ADMM on the School data. In the table,  ``a1" $=$ AS-{\sc Ssnpal}, ``a2" $=$ AS-ADMM, ``a3" $=$ {\sc Ssnpal} and ``a4" $=$ ADMM. We terminate algorithms when $R_{{\rm kkt}}\le 10^{-3}$ and $R_{{\rm kkt}}\le 10^{-6}$, respectively.} \label{TableSchool:p_e36}\\
				
				\hline
				\multirow{2}*{$\gamma$} &
				\multicolumn{4}{c|}{Time} &
				\multicolumn{4}{c|}{$R_ {kkt}$}\\ \cline{2-9}
				&  \multicolumn{1}{c}{a1}& \multicolumn{1}{c}{a2}& \multicolumn{1}{c}{a3}& \multicolumn{1}{c|}{a4}
				&\multicolumn{1}{c}{a1}& \multicolumn{1}{c}{a2}& \multicolumn{1}{c}{a3}& \multicolumn{1}{c|}{a4} \\ 
				\hline
				\multicolumn{9}{|l|}{tol=1e-03} \\  \hline
				\endfirsthead
				
				\multicolumn{9}{c|}{{\bfseries \tablename\ \thetable{} -- continued from previous page}} \\
				
				\hline
				\multirow{2}*{$\gamma$} &
				\multicolumn{4}{c|}{Time} &
				\multicolumn{4}{c|}{$R_ {kkt}$}\\ \cline{2-9}
				&  \multicolumn{1}{c}{a1}& \multicolumn{1}{c}{a2}& \multicolumn{1}{c}{a3}& \multicolumn{1}{c|}{a4}
				&\multicolumn{1}{c}{a1}& \multicolumn{1}{c}{a2}& \multicolumn{1}{c}{a3}& \multicolumn{1}{c|}{a4}
				\\ 
				\endhead
				
				\multicolumn{9}{|r|}{\textbf{Continued on next page}} \\
				\hline
				\endfoot
				
				\hline
				\endlastfoot
				0.01	 & 0.02 	 & 0.05 	 & 2.53 	 & 7.68 	 & 4.3e-04	 & 8.3e-04	 & 5.5e-04	 & 7.5e-04\\  
\hline0.03	 & 0.08 	 & 0.05 	 & 2.45 	 & 7.68 	 & 3.6e-04	 & 9.0e-04	 & 5.5e-04	 & 7.5e-04\\  
\hline0.05	 & 0.05 	 & 0.03 	 & 2.44 	 & 7.93 	 & 3.1e-04	 & 1.0e-03	 & 5.5e-04	 & 7.4e-04 \\
\hline \multicolumn{9}{|l|}{tol=1e-06 } \\
\hline0.01	 & 0.09 	 & 0.06 	 & 2.68 	 & 59.75 	 & 1.1e-07	 & 8.2e-07	 & 9.4e-07	 & 9.6e-07\\  
\hline0.03	 & 0.06 	 & 0.06 	 & 2.66 	 & 60.55 	 & 9.4e-08	 & 8.7e-07	 & 9.4e-07	 & 9.5e-07\\  
\hline0.05	 & 0.07 	 & 0.06 	 & 2.64 	 & 61.14 	 & 8.0e-08	 & 9.5e-07	 & 8.9e-07	 & 9.4e-07
		\end{longtable}}
	\end{center}
	
	Table \ref{TableSchool:p_e36} presents the detailed performance of AS-{\sc Ssnpal}, {\sc Ssnpal}, AS-ADMM, and ADMM in solving the multi-task Lasso problem on the School dataset, considering three values of $\gamma$: $0.01$, $0.03$, and $0.05$. For each algorithm, the table reports computation time (Time) and relative KKT residuals ($R_{\rm kkt}$) for each instance. As shown, the computation time of AS-ADMM is occasionally comparable to that of our proposed algorithm, likely due to the small size of the School dataset. However, both AS-{\sc Ssnpal} and AS-ADMM significantly outperform {\sc Ssnpal} and ADMM in terms of running time. Specifically, for $\gamma=0.03$, AS-{\sc Ssnpal} and AS-ADMM take only about $0.06$ seconds to achieve a relative KKT residual of $R_{\rm kkt} \leq 10^{-6}$, whereas {\sc Ssnpal} requires over $2.5$ seconds, and ADMM takes as long as $60.55$ seconds. These results demonstrate that the adaptive sieving strategy substantially enhances the efficiency of the {\sc Ssnpal} and ADMM algorithms for solving the multi-task Lasso problem on the School dataset.

	\subsubsection{LIBSVM data}
	Next, we perform the numerical tests for the AS-{\sc Ssnpal}, {\sc Ssnpal}, AS-ADMM and ADMM algorithms on real data $(\widetilde{X},y)$ obtained from LIBSVM data repository. Table \ref{Tablesummary} provides the basic information for the collected instances. In the table, ``proname" refers to the problem name, while $m$, $d$, and $n$ represent the sample size, feature dimensionality, and number of tasks, respectively. For datasets where the number of tasks is unspecified, we have chosen $n=20$. Additionally, the vector $y$ and each row of $\widetilde{X}$ are normalized to have a unit $\ell_2$-norm.

	\begin{center}{\footnotesize}
		\setlength{\tabcolsep}{8pt} 
		\renewcommand\arraystretch{1.1}
		\begin{longtable}{|c|c|c|c|} 
			\caption{\footnotesize {Statistics of the LIBSVM test instances.}}\label{Tablesummary}  \\	
			\hline
			proname  &$m$ &$d$ &$n$  \\
			\hline
			\endfirsthead
			
			\multicolumn{4}{c}{{\bfseries \tablename\ \thetable{} -- continued from previous page}} \\
			
			\hline
			proname  &$m$ &$d$ &$n$  \\
			\hline
			\endhead
			
			\multicolumn{4}{|r|}{\textbf{Continued on next page}} \\
			\hline
			\endfoot
			
			\hline
			\endlastfoot			
			
 dna	 & $ 3186$	 & $ 180$	 & $ 20 $\\  
\hline
 mnist.test	 & $ 10000$	 & $ 778$	 & $ 20 $\\  
\hline
 mushrooms	 & $ 8124$	 & $ 112$	 & $ 20 $\\  
\hline
 protein	 & $ 24387$	 & $ 357$	 & $ 20 $\\  
\hline
 usps	 & $ 9298$	 & $ 256$	 & $ 20 $\\  
\hline
 w1a	 & $ 28844$	 & $ 300$	 & $ 20 $\\  
\hline
 w2a	 & $ 49749$	 & $ 300$	 & $ 20 $\\  
\hline
 w3a	 & $ 39863$	 & $ 300$	 & $ 20 $
		\end{longtable}
	\end{center}

	\begin{center}
		\setlength{\tabcolsep}{8.0pt}{
			\begin{longtable}{|c|c|cccc|cccc|}
				\captionsetup{width=0.9\textwidth}
				\caption{The performance of AS-{\sc Ssnpal}, AS-ADMM, {\sc Ssnpal} and ADMM on the LIBSVM data. In the table, we terminate algorithms when $R_{{\rm kkt}}\le 10^{-3}$.  ``a1"=AS-{\sc Ssnpal}, ``a2"=AS-ADMM, ``a3"={\sc Ssnpal}, ``a4"=ADMM. Times are shown in seconds.} \label{TableReal:p_e3}\\
				
				\hline 
				\multirow{2}*{i} &\multirow{2}*{$\gamma$}&
				\multicolumn{4}{c|}{Time} &
				\multicolumn{4}{c|}{$R_{kkt}$} \\ \cline{3-10}
				& &  \multicolumn{1}{c}{a1}& \multicolumn{1}{c}{a2}& \multicolumn{1}{c}{a3}& \multicolumn{1}{c|}{a4} &\multicolumn{1}{c}{a1}& \multicolumn{1}{c}{a2}& \multicolumn{1}{c}{a3}& \multicolumn{1}{c|}{a4}  \\ \hline
				\endfirsthead
				
				\multicolumn{10}{c}{{\bfseries \tablename\ \thetable{} -- continued from previous page}} \\
				
				\hline
				\multirow{2}*{i} &\multirow{2}*{$\gamma$}&
				\multicolumn{4}{c|}{Time} &
				\multicolumn{4}{c|}{$R_{kkt}$} \\ \cline{3-10}
				& &  \multicolumn{1}{c}{a1}& \multicolumn{1}{c}{a2}& \multicolumn{1}{c}{a3}& \multicolumn{1}{c|}{a4} &\multicolumn{1}{c}{a1}& \multicolumn{1}{c}{a2}& \multicolumn{1}{c}{a3}& \multicolumn{1}{c|}{a4}  \\ \hline
				\endhead
				
				\multicolumn{10}{|r|}{\textbf{Continued on next page}} \\
				\hline
				\endfoot
				
				\hline
				\endlastfoot
				
				\multirow{3}*{1}	 & 0.01 	 & 0.02 	 & 0.01 	 & 0.52 	 & 4.12 	 & 4.3e-04	 & 8.3e-04	 & 5.5e-04	 & 7.5e-04\\  
	 & 0.03 	 & 0.05 	 & 0.03 	 & 1.02 	 & 8.12 	 & 3.6e-04	 & 9.0e-04	 & 5.5e-04	 & 7.5e-04\\  
	 & 0.05 	 & 0.11 	 & 0.05 	 & 1.56 	 & 12.12 	 & 3.1e-04	 & 1.0e-03	 & 5.5e-04	 & 7.4e-04\\  
\hline\multirow{3}*{2}	 & 0.01 	 & 0.13 	 & 0.66 	 & 23.79 	 & 433.90 	 & 3.1e-04	 & 1.0e-03	 & 6.9e-04	 & 9.6e-04\\  
	 & 0.03 	 & 0.82 	 & 1.32 	 & 46.67 	 & 861.70 	 & 2.6e-04	 & 9.9e-04	 & 6.9e-04	 & 9.7e-04\\  
	 & 0.05 	 & 1.01 	 & 1.90 	 & 69.81 	 & 1292.36 	 & 9.9e-04	 & 9.9e-04	 & 7.0e-04	 & 1.0e-03\\  
\hline\multirow{3}*{3}	 & 0.01 	 & 0.05 	 & 0.28 	 & 0.61 	 & 4.41 	 & 5.3e-04	 & 1.0e-03	 & 7.1e-04	 & 8.6e-04\\  
	 & 0.03 	 & 0.12 	 & 0.53 	 & 1.17 	 & 8.82 	 & 3.7e-04	 & 1.0e-03	 & 7.0e-04	 & 8.6e-04\\  
	 & 0.05 	 & 0.21 	 & 0.76 	 & 1.66 	 & 13.30 	 & 4.2e-04	 & 9.9e-04	 & 7.0e-04	 & 8.6e-04\\  
\hline\multirow{3}*{4}	 & 0.01 	 & 0.09 	 & 0.30 	 & 11.65 	 & 66.14 	 & 2.3e-04	 & 1.0e-03	 & 6.7e-04	 & 8.0e-04\\  
	 & 0.03 	 & 0.20 	 & 0.63 	 & 23.71 	 & 132.29 	 & 9.8e-04	 & 9.9e-04	 & 6.7e-04	 & 8.0e-04\\  
	 & 0.05 	 & 0.49 	 & 0.93 	 & 35.81 	 & 198.69 	 & 8.9e-04	 & 9.9e-04	 & 6.7e-04	 & 8.1e-04\\  
\hline\multirow{3}*{5}	 & 0.01 	 & 0.06 	 & 0.28 	 & 2.66 	 & 41.23 	 & 7.8e-04	 & 9.9e-04	 & 6.9e-04	 & 8.3e-04\\  
	 & 0.03 	 & 0.17 	 & 0.56 	 & 5.33 	 & 82.63 	 & 6.4e-04	 & 9.9e-04	 & 6.9e-04	 & 8.3e-04\\  
	 & 0.05 	 & 0.23 	 & 0.81 	 & 7.95 	 & 124.18 	 & 5.5e-04	 & 9.9e-04	 & 7.1e-04	 & 8.8e-04\\  
\hline\multirow{3}*{6}	 & 0.01 	 & 0.43 	 & 0.48 	 & 10.46 	 & 83.59 	 & 3.8e-04	 & 1.0e-03	 & 7.2e-04	 & 8.8e-04\\  
	 & 0.03 	 & 0.54 	 & 0.94 	 & 20.64 	 & 167.85 	 & 3.6e-04	 & 1.0e-03	 & 7.2e-04	 & 8.8e-04\\  
	 & 0.05 	 & 0.65 	 & 1.36 	 & 30.84 	 & 255.19 	 & 3.3e-04	 & 1.0e-03	 & 7.2e-04	 & 8.9e-04\\  
\hline\multirow{3}*{7}	 & 0.01 	 & 0.14 	 & 0.94 	 & 14.39 	 & 81.18 	 & 6.1e-04	 & 1.0e-03	 & 7.2e-04	 & 8.4e-04\\  
	 & 0.03 	 & 0.30 	 & 1.68 	 & 30.38 	 & 160.96 	 & 3.8e-04	 & 1.0e-03	 & 7.2e-04	 & 8.4e-04\\  
	 & 0.05 	 & 0.45 	 & 2.39 	 & 45.55 	 & 234.72 	 & 3.6e-04	 & 1.0e-03	 & 7.2e-04	 & 8.4e-04\\  
\hline\multirow{3}*{8}	 & 0.01 	 & 0.16 	 & 0.83 	 & 10.97 	 & 72.37 	 & 3.9e-04	 & 1.0e-03	 & 7.1e-04	 & 8.4e-04\\  
	 & 0.03 	 & 0.30 	 & 1.55 	 & 21.89 	 & 139.72 	 & 2.2e-04	 & 9.9e-04	 & 7.1e-04	 & 8.4e-04\\  
	 & 0.05 	 & 0.42 	 & 2.09 	 & 32.88 	 & 207.94 	 & 2.1e-04	 & 1.0e-03	 & 7.1e-04	 & 8.4e-04
				
		\end{longtable}}
	\end{center}

	\begin{center}
		\setlength{\tabcolsep}{8.0pt}{
			\begin{longtable}{|c|c|cccc|cccc|}
				\captionsetup{width=0.9\textwidth}
				\caption{The performance of AS-{\sc Ssnpal}, AS-ADMM, {\sc Ssnpal} and ADMM on the LIBSVM data. In the table, we terminate algorithms when $R_{{\rm kkt}}\le 10^{-6}$.  ``a1"=AS-{\sc Ssnpal}, ``a2"=AS-ADMM, ``a3"={\sc Ssnpal}, ``a4"=ADMM. Times are shown in seconds.} \label{TableReal:p_e6}\\
				
				\hline 
				\multirow{2}*{i} &\multirow{2}*{$\gamma$}&
				\multicolumn{4}{c|}{Time} &
				\multicolumn{4}{c|}{$R_{kkt}$} \\ \cline{3-10}
				& &  \multicolumn{1}{c}{a1}& \multicolumn{1}{c}{a2}& \multicolumn{1}{c}{a3}& \multicolumn{1}{c|}{a4} &\multicolumn{1}{c}{a1}& \multicolumn{1}{c}{a2}& \multicolumn{1}{c}{a3}& \multicolumn{1}{c|}{a4}  \\ \hline
				\endfirsthead
				
				\multicolumn{10}{c}{{\bfseries \tablename\ \thetable{} -- continued from previous page}} \\
				
				\hline
				\multirow{2}*{i} &\multirow{2}*{$\gamma$}&
				\multicolumn{4}{c|}{Time} &
				\multicolumn{4}{c|}{$R_{kkt}$} \\ \cline{3-10}
				& &  \multicolumn{1}{c}{a1}& \multicolumn{1}{c}{a2}& \multicolumn{1}{c}{a3}& \multicolumn{1}{c|}{a4} &\multicolumn{1}{c}{a1}& \multicolumn{1}{c}{a2}& \multicolumn{1}{c}{a3}& \multicolumn{1}{c|}{a4}  \\ \hline
				\endhead
				
				\multicolumn{10}{|r|}{\textbf{Continued on next page}} \\
				\hline
				\endfoot
				
				\hline
				\endlastfoot
				
				\multirow{3}*{1}	 & 0.01 	 & 0.06 	 & 0.02 	 & 0.86 	 & 40.28 	 & 1.1e-07	 & 8.2e-07	 & 9.4e-07	 & 9.6e-07\\  
	 & 0.03 	 & 0.08 	 & 0.05 	 & 1.74 	 & 82.05 	 & 9.4e-08	 & 8.7e-07	 & 9.4e-07	 & 9.5e-07\\  
	 & 0.05 	 & 0.14 	 & 0.08 	 & 2.65 	 & 135.01 	 & 8.0e-08	 & 9.5e-07	 & 8.9e-07	 & 9.4e-07\\  
\hline\multirow{3}*{2}	 & 0.01 	 & 0.48 	 & 1.56 	 & 45.83 	 & 7230.77 	 & 4.7e-07	 & 9.9e-07	 & 4.9e-07	 & \textbf{4.7e-05}\\  
	 & 0.03 	 & 0.86 	 & 3.22 	 & 97.48 	 & 14462.33 	 & 3.9e-07	 & 9.9e-07	 & 3.9e-07	 & \textbf{4.1e-05}\\  
	 & 0.05 	 & 1.08 	 & 4.82 	 & 143.87 	 & 21691.53 	 & 3.3e-07	 & 9.9e-07	 & 1.0e-06	 & \textbf{4.1e-05}\\  
\hline\multirow{3}*{3}	 & 0.01 	 & 0.13 	 & 0.59 	 & 1.08 	 & 426.80 	 & 8.0e-08	 & 1.0e-06	 & 3.4e-07	 & 1.0e-06\\  
	 & 0.03 	 & 0.19 	 & 1.16 	 & 2.01 	 & 871.86 	 & 6.6e-08	 & 9.9e-07	 & 3.3e-07	 & 1.0e-06\\  
	 & 0.05 	 & 0.23 	 & 1.72 	 & 3.00 	 & 1318.81 	 & 8.9e-07	 & 1.0e-06	 & 3.2e-07	 & 1.0e-06\\  
\hline\multirow{3}*{4}	 & 0.01 	 & 0.25 	 & 0.68 	 & 18.82 	 & 2757.20 	 & 2.9e-07	 & 1.0e-06	 & 7.6e-07	 & 9.9e-07\\  
	 & 0.03 	 & 0.50 	 & 1.37 	 & 36.40 	 & 5632.78 	 & 2.7e-07	 & 1.0e-06	 & 7.7e-07	 & 1.0e-06\\  
	 & 0.05 	 & 0.63 	 & 2.06 	 & 54.22 	 & 8549.62 	 & 2.4e-07	 & 9.8e-07	 & 7.8e-07	 & 1.0e-06\\  
\hline\multirow{3}*{5}	 & 0.01 	 & 0.17 	 & 0.59 	 & 4.22 	 & 3639.15 	 & 1.2e-07	 & 9.8e-07	 & 8.7e-07	 & 1.0e-06\\  
	 & 0.03 	 & 0.25 	 & 1.10 	 & 8.62 	 & 7286.30 	 & 9.8e-08	 & 9.9e-07	 & 8.2e-07	 & 1.0e-06\\  
	 & 0.05 	 & 0.32 	 & 1.59 	 & 13.19 	 & 10939.24 	 & 8.3e-08	 & 9.9e-07	 & 7.8e-07	 & 1.0e-06\\  
\hline\multirow{3}*{6}	 & 0.01 	 & 0.28 	 & 3.57 	 & 17.25 	 & 7204.02 	 & 2.3e-07	 & 1.0e-06	 & 7.8e-07	 & \textbf{1.5e-06}\\  
	 & 0.03 	 & 0.58 	 & 6.92 	 & 37.39 	 & 14406.67 	 & 1.9e-07	 & 1.0e-06	 & 7.1e-07	 & \textbf{1.5e-06}\\  
	 & 0.05 	 & 1.38 	 & 10.17 	 & 61.89 	 & 21618.24 	 & 3.0e-07	 & 1.0e-06	 & 6.8e-07	 & \textbf{1.5e-06}\\  
\hline\multirow{3}*{7}	 & 0.01 	 & 0.34 	 & 8.57 	 & 24.38 	 & 7207.05 	 & 7.6e-07	 & 1.0e-06	 & 7.2e-07	 & \textbf{1.5e-06}\\  
	 & 0.03 	 & 0.85 	 & 15.93 	 & 55.35 	 & 14413.24 	 & 7.9e-08	 & 1.0e-06	 & 5.9e-07	 & \textbf{1.3e-06}\\  
	 & 0.05 	 & 2.02 	 & 23.61 	 & 90.24 	 & 21615.35 	 & 1.4e-07	 & 1.0e-06	 & 5.9e-07	 & \textbf{1.3e-06}\\  
\hline\multirow{3}*{8}	 & 0.01 	 & 0.32 	 & 6.85 	 & 19.06 	 & 7203.79 	 & 3.6e-07	 & 1.0e-06	 & 5.8e-07	 & \textbf{1.2e-06}\\  
	 & 0.03 	 & 0.71 	 & 13.43 	 & 44.97 	 & 14366.67 	 & 7.0e-07	 & 1.0e-06	 & 4.8e-07	 & 1.0e-06\\  
	 & 0.05 	 & 2.20 	 & 20.15 	 & 72.47 	 & 21469.26 	 & 7.0e-08	 & 1.0e-06	 & 4.4e-07	 & 1.0e-06
				
		\end{longtable}}
	\end{center}
	
	Table \ref{TableReal:p_e3} reports the numerical results of the AS-{\sc Ssnpal}, AS-ADMM, {\sc Ssnpal} and ADMM algorithms when $R_{{\rm kkt}}\le 10^{-3}$. Similarly, ``Time" refers to the cumulative computation time of the tested algorithm for generating solution paths of Problem (\ref{mtlproblem1}). Table  \ref{TableReal:p_e3} demonstrates that the AS-{\sc Ssnpal} algorithm outperforms the other algorithms tested. For instance, when addressing the $i=6$ instance with moderate dimensions $(m,d,n) = (28844,300,20)$, AS-{\sc Ssnpal} is approximately twice as fast as AS-ADMM, $47$ times faster than {\sc Ssnpal}, and $392$ times faster than ADMM. Further comparisons of AS-{\sc Ssnpal} with AS-ADMM, {\sc Ssnpal}, and ADMM for solving multi-task Lasso problems, under the condition that $R_{{\rm kkt}}\le 10^{-6}$, are presented in Table \ref{TableReal:p_e6}. In terms of robustness, AS-{\sc Ssnpal}, AS-ADMM, and {\sc Ssnpal} consistently solve all instances, whereas ADMM fails on some instances. Moreover, AS-{\sc Ssnpal} is notably more efficient than the other algorithms, especially ADMM. Specifically, AS-{\sc Ssnpal} is at least eight times faster than AS-ADMM, $13$ times faster than {\sc Ssnpal}, and $5730$ times faster than ADMM on successful instances. These findings further confirm that AS-{\sc Ssnpal} is both more effective and more robust for solving the multi-task Lasso problem.

	\section{Conclusion} \label{sec:conclude} 
	In this paper, we propose an adaptive sieving (AS) strategy integrated with the semismooth Newton proximal augmented Lagrangian ({\sc Ssnpal}) algorithm to generate solution path for multi-task Lasso problems. Specifically, the {\sc Ssnpal} algorithm is used to solve the reduced problems within the AS strategy, exhibiting both global convergence and an asymptotically superlinear convergence rate. For solving the inner problems in the {\sc Ssnpal} algorithm, we introduce a semismooth Newton method, which is proven to achieve superlinear or even quadratic convergence rates. Additionally, we show that the AS strategy terminates within a finite number of iterations. Finally, we evaluate the numerical performance of the AS-{\sc Ssnpal} algorithm against other algorithms (AS-ADMM, {\sc Ssnpal}, and ADMM) using both synthetic and real datasets. The results highlight the superior efficiency and robustness of our proposed algorithm, particularly on real datasets.

	\acks{The research of Lanyu Lin was supported by the Educational Research Project for Young and Middle-aged Teachers of Fujian Province (Grant No. JAT241003). The research of Yong-Jin Liu was supported by grants from the National Natural Science Foundation of China (Grant No. 12271097), the Key Program of National Science Foundation of Fujian Province of China (Grant No. 2023J02007), the Central Guidance on Local Science and Technology Development Fund of Fujian Province (Grant No. 2023L3003), and the Fujian Alliance of Mathematics (Grant No. 2023SXLMMS01). The research of Junfeng Yang was supported in part by the National Natural Science Foundation of China (Grant No. 12371301) and the Natural Science Foundation for Distinguished Young Scholars of Gansu Province (Grant No. 22JR5RA223).}


\vskip 0.2in	
	\begin{thebibliography}{99}
		
		
		\bibitem{Argyriou2008Convex}
		Argyriou, A., Evgeniou, T., Pontil, M.: Convex multi-task feature learning. Machine Learning. \textbf{73}(3), 243--272 (2008)
		
		\bibitem{Caruana1998Multitask}
		Caruana, R.: Multitask learning. Machine Learning. \textbf{28}(1), 41--75 (1997)	
		
		\bibitem{Chen2011Integrating}
		Chen, J.H., Zhou, J.Y., Ye, J.P.: Integrating low-rank and group-sparse structures for robust multi-task learning. In the 17th ACM SIGKDD International Conference on Knowledge Discovery and Data Mining. 42--50 (2011) 
		
		\bibitem{Chen2009Accelerated}
		Chen, X., Pan, W., Kwok, J.T., Carbonell, J.G.: Accelerated gradient method for multi-task sparse learning problem. In 19th IEEE International Conference on Data Mining. 746--751 (2009)
		
		\bibitem{Chu2020Semismooth}	
		Chu, D.J., Zhang, C.S., Sun, S.L., Tao, T.: Semismooth Newton algorithm for efficient projections onto $\ell_{1,\infty}$-norm ball. International Conference on Machine Learning. 1974--1983 (2020)
		
		\bibitem{Clarke1983Optimization}
		Clarke, F.H.: Optimization and Nonsmooth Analysis. Wiley, New York (1983)
		
		\bibitem{Dennis1983Numerical}
		Dennis, J.E., Schnabel, R.B.:  Numerical Methods for Unconstrained Optimization and Nonlinear Equations. Prentice-Hall, New Jersey (1983)
		
		\bibitem{Du2015An}
		Du, M.: An inexact alternating direction method of multipliers for convex composite conic programming with nonlinear constraints. Ph.D. thesis, National University of Singapore (2015)
		
		\bibitem{Duchi2009Efficient}
		Duchi, J., Singer, Y.: Efficient online and batch learning using forward backward splitting. Journal of Machine Learning Research. \textbf{10}(18), 2899--2934 (2009)
		
		\bibitem{Facchinei1995Minimization}
		Facchinei, F.: Minimization of SC$^1$ functions and the Maratos effect. Operations Research Letters. \textbf{17}(3), 131--137 (1995)
		
		\bibitem{Facchinei2003Finite}
		Facchinei, F., Pang, J.-S.: Finite-Dimensional Variational Inequalities and Complementarity Problems. Springer, New York (2003)
		
		\bibitem{Fang2021Efficient}
		Fang, S., Liu, Y.-J., Xiong, X.Z.: Efficient sparse Hessian-based semismooth Newton algorithms for Dantzig selector. SIAM Journal on Scientific Computing. \textbf{43}(6), A4147--A4171 (2021) 
		
		\bibitem{Gabay1976A}
		Gabay, D., Mercier, B.: A dual algorithm for the solution of nonlinear variational problems via finite element approximation. Computers Mathematics with Applications. \textbf{2}(1), 17--40 (1976)
		
		\bibitem{Ghaoui2010Safe}
		Ghaoui, L.E., Viallon, V., Rabbani, T.: Safe feature elimination in sparse supervised learning. Technical report, University of California, Berkeley (2010)
		
		\bibitem{Glowinski1975Sur}
		Glowinski, R., Marroco, A.: Sur l'approximation, par ${\rm \acute{e}}$l${\rm \acute{e}}$ments finis d'ordre un, et la r${\rm \acute{e}}$solution, par p${\rm \acute{e}}$nalisation-dualit${\rm \acute{e}}$ d'une classe de probl{\rm\`{e}}mes de Dirichlet non lin${\rm \acute{e}}$aires. Revue Francaise d’Automatique, Informatique, et Recherche Op\`{e}rationelle. \textbf{9}, 41--76 (1975)
		
		
		\bibitem{Han1997Newton}
		Han, J.Y., Sun, D.F.: Newton and quasi-Newton methods for normal maps with polyhedral sets. Journal of Optimization Theory and Applications. \textbf{94}(3), 659--676 (1997)
		
		\bibitem{Klatte2002Nonsmooth}
		Klatte, D., Kummer, B.: Nonsmooth Equations in Optimization: Regularity, Calculus, Methods and Applications. Dordrecht: Kluwer Academic Publishers (2002)
		
		
		\bibitem{Li2023MARS}
		Li, Q., Jiang, B.Y., Sun, D.F.: MARS: A second-order reduction algorithm for high-dimensional sparse precision matrices estimation. Journal of Machine Learning Research. \textbf{24}(134), 1--44 (2023)
		
		\bibitem{Li2018On}
		Li, X.D., Sun, D.F., Toh, K.-C.: On efficiently solving the subproblems of a level-set method for fused Lasso problems. SIAM Journal on Optimization. \textbf{28}(2), 1842--1866 (2017) 
		
		\bibitem{Li2018A}
		Li, X.D., Sun, D.F., Toh, K.-C.: A highly efficient semismooth Newton augmented Lagrangian method for solving Lasso problems. SIAM Journal on Optimization. \textbf{28}(1), 433--458 (2018)
		
		\bibitem{Li2020On}
		Li, X.D., Sun, D.F., Toh, K.-C.: On the efficient computation of a generalized Jacobian of the projector over the Birkhoff polytope. Mathematical Programming. \textbf{179}(1), 419--446 (2020)
		
		\bibitem{Lin2024An}
		Lin, L.Y., Liu, Y.-J.: An inexact semismooth Newton-based augmented lagrangian algorithm for multi-task Lasso problems. Asia-Pacific Journal of Operational Research. \textbf{41}(3), 1--26 (2024)
		
		
		\bibitem{Lin2019Efficient}
		Lin, M.X., Liu, Y.-J., Sun, D.F., Toh, K.-C.: Efficient sparse semismooth newton methods for the clustered Lasso problem. SIAM Journal on Optimization. \textbf{29}(3), 2026--2052 (2019)
		
		\bibitem{Lin2020Adaptive}
		Lin, M.X., Yuan, Y.C., Sun, D.F., Toh, K.-C.: Adaptive sieving with PPDNA: generating solution paths of exclusive Lasso models. arXiv preprint arXiv:2009.08719. (2020)
		
		\bibitem{Liu2009Blockwise}
		Liu, H., Palatucci, M., Zhang, J.: Blockwise coordinate descent procedures for the multi-task Lasso, with applications to neural semantic basis discovery. Proceedings of the 26th Annual International Conference on Machine Learning. 649--656 (2009)  
		
		\bibitem{Liu2023Dual}
		Liu, Y.-J., Zhou, W.M.: Dual Newton proximal point algorithm for solution paths of the $\ell_1$-Regularized logistic regression. arxiv preprint arxiv:2310.19353. (2023)
		
		
		\bibitem{Luque1984Asymptotic}
		Luque, F.J.: Asymptotic convergence analysis of the proximal point algorithm. SIAM Journal on Control and Optimization. \textbf{22}(2), 277--293 (1984)
		
		
		\bibitem{Moreau1965Proximite}
		Moreau, J.-J.: Proximit\'{e} et dualit\'{e} dans un espace hilbertien. Bulletin de la Soci\'{e}t\'{e} math\'{e}matique de France. \textbf{93}, 273--299 (1965)
		
		\bibitem{Obozinski2006Multi}
		Obozinski, G., Taskar, B., Jordan, M.: Multi-task feature selection. Technical report, UC Berkeley (2006)
		
		\bibitem{Obozinskis2008High}
		Obozinski, G., Wainwright, M.J., Jordan, M.I.: High-dimensional support union recovery in multivariate regression. Advances in Neural Information Processing Systems. 1217--1224 (2008) 
		
		\bibitem{Osborne2000On}
		Osborne, M.R., Presnell, B., Turlach, B.A.: On the Lasso and its dual. Journal of Computational and Graphical statistics. \textbf{93}(2), 319--337 (2000)
		
		\bibitem{Parameswaran2010Large}
		Parameswaran, S., Weinberger, K.Q.: Large margin multi-task metric learning. \textbf{23}, 1867--1875 (2010)	
		
		\bibitem{Quadrianto2010Multitask}
		Quadrianto, N., Smola, A., Caetano, T., Vishwanathan, S., Petterson, J.: Multitask learning without label correspondences. Advances in Neural Information Processing Systems. \textbf{23}, (2010) 
		
		\bibitem{Quattoni2009An}
		Quattoni, A., Carreras, X., Collins, M., Darrell, T.: An efficient projection for $\ell_{1,\infty}$ regularization. Proceedings of the 26th Annual International Conference on Machine Learning.  857--864 (2009)
		
		
		\bibitem{Rockafellar1974Conjugate}
		Rockafellar, R.T.: Conjugate Duality and Optimization. SIAM, Philadelphia (1974)
		
		\bibitem{Rockafellar1976Augmented}
		Rockafellar, R.T.: Augmented Lagrangians and applications of the proximal point algorithm in convex programming. Mathematics of Operations Research. \textbf{1}(2), 97--116 (1976)
		
		\bibitem{Rockafellar1976Monotone}
		Rockafellar, R.T.: Monotone operators and the proximal point algorithm. SIAM Journal on Control and Optimization. \textbf{14}(5), 877--898 (1976)
		
		
		\bibitem{Rockafellar1998Variational}
		Rockafellar, R.T., Wets, R.J.-B.: Variational Analysis. Springer, Berlin (1998)
		
		
		
		\bibitem{Tibshirani2012Strong}
		Tibshirani, R., Bien, J., Friedman, J., Hastie, T., Simon, N., Taylor, J., Tibshirani, R.J.: Strong rules for discarding predictors in Lasso-type problems. Journal of the Royal Statistical Society Series B: Statistical Methodology. \textbf{74}(2), 245--266 (2012)
		
		\bibitem{Tropp2006Algorithms}
		Tropp, J.A., Gilbert, A.C., Strauss, M.J.: Algorithms for simultaneous sparse approximation. Part II: Convex relaxation. Signal Processing. \textbf{86}(3), 589--602 (2006)
		
		\bibitem{Turlach2005Simultaneous}
		Turlach, B.A., Venables, W.N., Wright, S.J.: Simultaneous variable selection. Technometrics. \textbf{27}, 349--363 (2005)
		
		\bibitem{Wang2013Lasso}
		Wang, J, Zhou, J.Y., Wonka, P. Ye, J.P.: Lasso screening rules via dual polytope projection. Advances in Neural Information Processing Systems. \textbf{26}, 1--9 (2013) 
		
		\bibitem{Wright1997Primal}
		Wright, S.J.: Primal-Dual Interior-Point Methods. SIAM, Philadelphia (1997)
		
		\bibitem{Zhang2006A}
		Zhang, J.: A probabilistic framework for multi-task learning. Ph.D. dissertation, Carnegie Mellon University (2006)
		
		\bibitem{Zhao2010A}
		Zhao, X.-Y., Sun, D.F., Toh, K.-C.: A Newton-CG augmented Lagrangian method for semidefinite programming. SIAM Journal on Optimization. \textbf{20}(4), 1737--1765 (2010)
		
		\bibitem{Zhou2011A}
		Zhou, J.Y., Yuan, L., Liu, J., Ye, J.P.:  A multi-task learning formulation for predicting disease progression. Proceedings of the 17th ACM SIGKDD International Conference on Knowledge Discovery and Data Mining. 814--822 (2011) 
		
	\end{thebibliography}
\end{document}